\documentclass{amsart}
\usepackage{amsthm}
\usepackage{amssymb}
\usepackage{amsmath}
\usepackage{amsmath,amscd}
\usepackage{color}
\usepackage{hyperref}
\usepackage[all,arc]{xy}
\usepackage{graphicx}
\graphicspath{ {images/} }

\newtheorem{thm}{Theorem}[section]
\newtheorem{cor}[thm]{Corollary}

\newtheorem{lem}[thm]{Lemma}

\theoremstyle{definition}
\newtheorem{defn}[thm]{Definition}

\newtheorem{exmp}[thm]{Example}

\theoremstyle{remark}
\newtheorem{rem}[thm]{Remark}

\usepackage{tikz}
\usepgflibrary{shapes.geometric}
\usetikzlibrary{arrows}

\tikzstyle{V}=[draw, fill =black, circle, inner sep=0pt, minimum size=1.5pt]
\tikzstyle{wV}=[draw, fill =white, circle, inner sep=0pt, minimum size=4.5pt]
\tikzstyle{bV}=[draw, fill =black, circle, inner sep=0pt, minimum size=4.5pt]
\tikzstyle{over}=[draw=white,double=black,line width=2pt, double distance=.5pt]

\newcommand{\nc}{\newcommand}
\nc{\rnc}{\renewcommand}
\nc{\bb}[1]{{\mathbb #1}}
\nc{\bbA}{\bb{A}}\nc{\bbB}{\bb{B}}\nc{\bbC}{\bb{C}}\nc{\bbD}{\bb{D}}
\nc{\bbE}{\bb{E}}\nc{\bbF}{\bb{F}}\nc{\bbG}{\bb{G}}\nc{\bbH}{\bb{H}}
\nc{\bbI}{\bb{I}}\nc{\bbJ}{\bb{J}}\nc{\bbK}{\bb{K}}\nc{\bbL}{\bb{L}}
\nc{\bbM}{\bb{M}}\nc{\bbN}{\bb{N}}\nc{\bbO}{\bb{O}}\nc{\bbP}{\bb{P}}
\nc{\bbQ}{\bb{Q}}\nc{\bbR}{\bb{R}}\nc{\bbS}{\bb{S}}\nc{\bbT}{\bb{T}}
\nc{\bbU}{\bb{U}}\nc{\bbV}{\bb{V}}\nc{\bbW}{\bb{W}}\nc{\bbX}{\bb{X}}
\nc{\bbY}{\bb{Y}}\nc{\bbZ}{\bb{Z}}
\nc{\mbf}[1]{{\mathbf #1}}
\nc{\bfA}{\mbf{A}}\nc{\bfB}{\mbf{B}}\nc{\bfC}{\mbf{C}}\nc{\bfD}{\mbf{D}}
\nc{\bfE}{\mbf{E}}\nc{\bfF}{\mbf{F}}\nc{\bfG}{\mbf{G}}\nc{\bfH}{\mbf{H}}
\nc{\bfI}{\mbf{I}}\nc{\bfJ}{\mbf{J}}\nc{\bfK}{\mbf{K}}\nc{\bfL}{\mbf{L}}
\nc{\bfM}{\mbf{M}}\nc{\bfN}{\mbf{N}}\nc{\bfO}{\mbf{O}}\nc{\bfP}{\mbf{P}}
\nc{\bfQ}{\mbf{Q}}\nc{\bfR}{\mbf{R}}\nc{\bfS}{\mbf{S}}\nc{\bfT}{\mbf{T}}
\nc{\bfU}{\mbf{U}}\nc{\bfV}{\mbf{V}}\nc{\bfW}{\mbf{W}}\nc{\bfX}{\mbf{X}}
\nc{\bfY}{\mbf{Y}}\nc{\bfZ}{\mbf{Z}}
\nc{\bfa}{\mbf{a}}\nc{\bfb}{\mbf{b}}\nc{\bfc}{\mbf{c}}\nc{\bfd}{\mbf{d}}
\nc{\bfe}{\mbf{e}}\nc{\bff}{\mbf{f}}\nc{\bfg}{\mbf{g}}\nc{\bfh}{\mbf{h}}
\nc{\bfi}{\mbf{i}}\nc{\bfj}{\mbf{j}}\nc{\bfk}{\mbf{k}}\nc{\bfl}{\mbf{l}}
\nc{\bfm}{\mbf{m}}\nc{\bfn}{\mbf{n}}\nc{\bfo}{\mbf{o}}\nc{\bfp}{\mbf{p}}
\nc{\bfq}{\mbf{q}}\nc{\bfr}{\mbf{r}}\nc{\bfs}{\mbf{s}}\nc{\bft}{\mbf{t}}
\nc{\bfu}{\mbf{u}}\nc{\bfv}{\mbf{v}}\nc{\bfw}{\mbf{w}}\nc{\bfx}{\mbf{x}}
\nc{\bfy}{\mbf{y}}\nc{\bfz}{\mbf{z}}

\nc{\mcal}[1]{{\mathcal #1}}
\nc{\calA}{\mcal{A}}\nc{\calB}{\mcal{B}}\nc{\calC}{\mcal{C}}\nc{\calD}{\mcal{D}}
\nc{\calE}{\mcal{E}} \nc{\calF}{\mcal{F}}\nc{\calG}{\mcal{G}}\nc{\calH}{\mcal{H}}
\nc{\calI}{\mcal{I}}\nc{\calJ}{\mcal{J}}\nc{\calK}{\mcal{K}}\nc{\calL}{\mcal{L}}
\nc{\calM}{\mcal{M}}\nc{\calN}{\mcal{N}}\nc{\calO}{\mcal{O}}\nc{\calP}{\mcal{P}}
\nc{\calQ}{\mcal{Q}}\nc{\calR}{\mcal{R}}\nc{\calS}{\mcal{S}}\nc{\calT}{\mcal{T}}
\nc{\calU}{\mcal{U}}\nc{\calV}{\mcal{V}}\nc{\calW}{\mcal{W}}\nc{\calX}{\mcal{X}}
\nc{\calY}{\mcal{Y}}\nc{\calZ}{\mcal{Z}}
\nc{\fA}{\frak{A}}\nc{\fB}{\frak{B}}\nc{\fC}{\frak{C}} \nc{\fD}{\frak{D}}
\nc{\fE}{\frak{E}}\nc{\fF}{\frak{F}}\nc{\fG}{\frak{G}}\nc{\fH}{\frak{H}}
\nc{\fI}{\frak{I}}\nc{\fJ}{\frak{J}}\nc{\fK}{\frak{K}}\nc{\fL}{\frak{L}}
\nc{\fM}{\frak{M}}\nc{\fN}{\frak{N}}\nc{\fO}{\frak{O}}\nc{\fP}{\frak{P}}
\nc{\fQ}{\frak{Q}}\nc{\fR}{\frak{R}}\nc{\fS}{\frak{S}}\nc{\fT}{\frak{T}}
\nc{\fU}{\frak{U}}\nc{\fV}{\frak{V}}\nc{\fW}{\frak{W}}\nc{\fX}{\frak{X}}
\nc{\fY}{\frak{Y}}\nc{\fZ}{\frak{Z}}
\nc{\fa}{\frak{a}}\nc{\fb}{\frak{b}}\nc{\fc}{\frak{c}} \nc{\fd}{\frak{d}}
\nc{\fe}{\frak{e}}\nc{\fFf}{\frak{f}}\nc{\fg}{\frak{g}}\nc{\fh}{\frak{h}}
\nc{\fri}{\frak{i}}\nc{\fj}{\frak{j}}\nc{\fk}{\frak{k}}\nc{\fl}{\frak{l}}
\nc{\fm}{\frak{m}}\nc{\fn}{\frak{n}}\nc{\fo}{\frak{o}}\nc{\fp}{\frak{p}}
\nc{\fq}{\frak{q}}\nc{\fr}{\frak{r}}\nc{\fs}{\frak{s}}\nc{\ft}{\frak{t}}
\nc{\fu}{\frak{u}}\nc{\fv}{\frak{v}}\nc{\fw}{\frak{w}}\nc{\fx}{\frak{x}}
\nc{\fy}{\frak{y}}\nc{\fz}{\frak{z}}

\newcommand{\changjian}[1]{\textcolor{red}{ changjian: #1 }}
\newcommand{\gufang}[1]{\textcolor{green}{ gufang: #1 }}

\nc{\al}{\alpha}
\nc{\ep}{\epsilon}
\nc{\la}{\lambda}
\nc{\be}{\beta}
\nc{\de}{\delta}
\nc{\na}{\nabla}
\nc{\nap}{{\na_+}}
\nc{\nam}{{\na_-}}
\nc{\lefta}{\leftarrow}
\nc{\wt}{\widetilde}

\DeclareMathOperator{\stab}{stab}

\DeclareMathOperator{\Attr}{Attr}
\DeclareMathOperator{\FAttr}{FAttr}
\DeclareMathOperator{\supp}{supp}

\DeclareMathOperator{\Aut}{Aut}

\DeclareMathOperator{\pt}{pt}
\DeclareMathOperator{\Ext}{Ext}
\DeclareMathOperator{\Pic}{Pic}

\DeclareMathOperator{\rank}{rank}
\DeclareMathOperator{\SL}{SL}
\DeclareMathOperator{\Lie}{Lie}

\DeclareMathOperator{\opp}{opp}
\DeclareMathOperator{\Max}{Max}
\DeclareMathOperator{\Min}{Min}
\DeclareMathOperator{\gr}{gr}
\DeclareMathOperator{\Hilb}{Hilb}

\nc{\stp}[2]{\stab^{+, {#1}}_{#2}}   
\nc{\stm}[2]{\stab^{-,{#1}}_{#2}} 
\nc{\stabp}[1]{\stp{\na_-}{#1}} 
\nc{\stabm}[1]{\stm{\na_+}{#1}} 

\nc{\fstab}{\fs\ft\fa\fb}
\DeclareMathOperator{\ob}{b}
\DeclareMathOperator{\oR}{R}

\DeclareMathOperator{\aff}{aff}

\DeclareMathOperator{\Coh}{Coh}
\DeclareMathOperator{\Mod}{Mod}
\DeclareMathOperator{\Frac}{Frac}

\newcommand{\inj}{\hookrightarrow}

\def\angl#1{{\langle #1\rangle}}

\newcommand{\frakg}{\mathfrak{g}}

\newcommand{\ZZ}{\mathbb{Z}}
\newcommand{\Gm}{\mathbb{G}_m}

\DeclareMathOperator{\Hom}{Hom}

\title[Wall-crossings and categorification]{Wall-crossings and a categorification of $K$-theory stable bases of the Springer resolution}

\author[C.~Su]{Changjian~Su}
\email{changjiansu@gmail.com}
\address{Department of Mathematics, University of Toronto, 40 St. George St., Room 6290,
Toronto, ON, M5S 2E4, Canada}

\author[G.~Zhao]{Gufang~Zhao}
\email{gufangz@unimelb.edu.au}
\address{School of Mathematics and Statistics, The University of Melbourne, 813 Swanston Street, Parkville VIC 3010, Australia}

\author[C.~Zhong]{Changlong~Zhong}
\email{czhong@albany.edu}
\address{State University of New York at Albany, 1400 Washington Ave, ES 110, Albany, NY, 12222, USA}


\begin{document}

\begin{abstract}
We compare the $K$-theory stable bases of the Springer resolution associated to different affine Weyl alcoves. We prove that (up to relabelling) the change of alcoves operators are given by the Demazure-Lusztig operators in the affine Hecke algebra. We then show that these bases are categorified by the Verma modules  of the Lie algebra, under the localization of Lie algebras in positive characteristic of Bezrukavnikov, Mirkovi\'c, and Rumynin. 
As an application, we prove that the wall-crossing matrices of the $K$-theory stable bases coincide with the monodromy matrices of the quantum cohomology of the Springer resolution.
\end{abstract}
\setcounter{tocdepth}{1}
\maketitle

\section{Introduction}

For a symplectic resolutions $X$ with symplectic form $\omega$ in the sense of \cite{Kal09}, 
the stable basis defined in \cite{MO} has found more and more applications. It is the crucial ingredient in the work of Maulik and Okounkov in constructing geometrical $R$-matrices, satisfying the Yang--Baxter equations. These $R$-matrices induces a Yangian action on the cohomology of Nakajima quiver varieties, so that the quantum connection of the quiver variety is identified 
 with the trigonometric Casimir connection for the Yangian. 
In the case when $X$ is the cotangent bundle of the flag variety  (the Springer resolution) the stable basis is computed by the first-named author in \cite{Su17}. Using this, he computed the quantum connection of the cotangent bundle of partial flag varieties \cite{Su16} and generalized the result in \cite{BMO}. Pulling back the stable basis of the Springer resolution to the flag variety, one gets the Chern--Schwartz--MacPherson classes \cite{M} of the Schubert cells, see \cite{AMSS17, RV15}. Combining this with the fact that the stable basis are the characteristic cycles of certain $\calD$-modules on the flag variety, the authors in \cite{AMSS17} proved the non-equivariant version of the positivity conjecture in \cite{AM16}.

The $K$-theoretic generalization of the stable basis is also undergoing significant developments, a 
survey of which can be found in \cite{O15, O18}. 

The definition of the $K$-theory stable bases for a symplectic resolution $(X, \omega)$  depends on a choice of a chamber in the Lie algebra of a maximal torus $A$ in $\Aut(X, \omega)$, an alcove in $\Pic(X)_\bbQ$, as well as an orientation of the tangent bundle $TX$ (see e.g. \S~\ref{sec:recoll}). Here a hyperplane configuration in $\Pic(X)_\bbQ$ naturally occurs.  

It is an interesting question to find the change of bases matrices for stable bases associated to different choice of chambers or alcoves. The case when $X=\Hilb_n(\bbC^2)$ is studied in \cite{GN}. If $X$ is a Nakajima quiver variety, such change of bases matrices are certain trigonometric $R$-matrices, and induced an action of  certain quantum loop algebra on the $K$-theory of Nakajima quiver varieties \cite{OS, RTV}, generalizing the construction by Nakajima \cite{N01}.
These played a crucial role in calculations of quantum $K$-theory of Nakajima quiver varieties \cite{O, OS}. 
For example, they fully determined the two sets  of the difference equations in the quantum $K$-theory of  Nakajima quiver varieties.

In the case when $X=T^*\calB$, the Springer resolution, the chamber structure is given by the Weyl chambers, and the alcove structure is given by the affine Weyl alcoves in $\Lambda_\bbQ:=\Pic(T^*\calB)_\bbQ$ in the dual of the Cartan Lie algebra. That is, the hyperplane configuration occurring in the definition of stable bases are the affine coroot hyperplanes $H_{\al^\vee, n}$ for some coroot $\al^\vee$ and $n\in \bbZ$, with the affine Weyl group $W_{\aff}$ being the  group generated by reflections. 

In \cite{SZZ} we calculated the stable bases of $T^*\calB$ associated to the fundamental alcove using the affine Hecke algebra actions via convolution \cite{CG}. 
We  identified the stable basis (resp. the fixed point basis)  with the standard basis (resp. the Casselman basis) in the Iwahori-invariants of unramified principal series for the $p$-adic Langlands dual groups. With this, in \cite{SZZ} we provided a $K$-theoretic interpretation of the Macdonald's formula for the spherical function \cite{C80} and the Casselman--Shalika formula for the spherical Whittaker function \cite{CS}. Pulling back the stable basis to the flag variety from its cotangent bundle, we get the motivic Chern classes \cite{BSY} of the Schubert cells \cite{AMSS19, FRW}. This connection is used to prove a series of conjectures about the Casselman basis, see \cite{AMSS19, BN11, BN17}.  

The goal of the present paper is to calculate the bases associated to an arbitrary choice of alcove, and find the change of bases operators for the bases associated to two different alcoves. 
We also determine the effect of changing the chambers. 

As applications, we prove that the $K$-theory stable basis is categorified by the Verma modules of the  quantization of $T^*\calB$ in positive characteristic \cite{BMR1}. Consequently, we prove that the change of bases operators induced by two different choices of alcoves coincide with the monodromy matrices of the quantum cohomology. 
These facts have been conjectured by Bezrukavnikov and Okounkov to hold for a general symplectic resolution $(X,\omega)$ \cite{O15}. The case when $X$ is the Nakajima quiver variety is proved in  \cite{BO} via a different method. However, to the best of our knowledge, the present paper provides one of the first examples where a connection between $K$-theory stable bases and quantizations in positive characteristic is established.

We now state our main result on the wall-crossings, without introducing too many notations and technicalities. 
For simplicity we fix the choice of the Weyl chamber to be the positive Weyl chamber $+$, and the orientation to be a canonical one given by the tangent bundle \S~\ref{sec:recoll}. For each choice of alcove $\na$, the basis elements are labeled by the torus fixed points on $T^*\calB$, which in turn are labeled by  elements of the Weyl group $W$. Hence the stable basis is denoted by $\{\stab^{+,T\calB,\na}_w\mid w\in W\}$.

Let $A$ be the maximal torus of $G$, and $A^\vee$ its Langlands dual. We use subindex $reg$ to denote the open complement to the coroot hyperplanes. Let $B'_{\aff}=\pi_1(A^\vee_{reg}/W)$. 
There is a quotient  $\bbC[q^{\pm1/2}][B'_{\aff}]\to \bbH$ from the group algebra to the affine Hecke algebra, and an action of $\bbH$ on $K_{A\times\bbC^*}(T^*\calB)$, constructed by Kazhdan-Lusztig \cite{KL} and Ginzburg \cite{CG}. This action describes the wall-crossings of stable bases completely.  This is the first main result in the present paper. See Theorem \ref{thm:wallcrossingHecke} for a precise statement.
\begin{thm}\label{Thm1:Intr}
	Assume two alcoves $\na_1, \na_2$  are adjacent and separated  by $H_{\al^\vee, n}$ for some positive coroot $\al^\vee$ and $n\in\bbZ$, and that $(\la_1, \al^\vee)<(\la_2, \al^\vee)$ for any $\la_i\in \na_i$. Let $b_{\na_1,\na_2}$ be the element in $\pi_1(\fh^*_{\bbC,reg}/W_{\aff})\subseteq B'_{\aff}$ represented by a positive path going from $\na_1$ to $ \na_2$. Then, the image of $b_{\na_2,\na_1}^{-1}$ in $\bbH$ sends the set $\{\calL_{-\rho}\otimes\stab^{+,T\calB,\na_1}_w\mid w\in W\}$ to $\{\calL_{-\rho}\otimes\stab^{+,T\calB,\na_2}_w\mid w\in W\}$, up to some (explicitly determined) scalars. 
\end{thm}
As an immediate consequence of Theorem~\ref{Thm1:Intr}, we proved that the standard objects in the quantizations of $T^*\calB$ in positive characteristic categorify the stable bases. 
More precisely, we have a categorification of the stable basis, well defined for the Springer resolution $T^*\calB_{\bbZ}$ over $\bbZ$. When base changing to characteristic zero and passing to the Grothendieck group, we get the usual stable basis. When base changing to positive characteristic, we get localization of Verma modules via the equivalence \cite{BM,BMR1, BMR2}.

\begin{thm}[Theorems~\ref{thm:IntStable} and   \ref{Thm:cat}]
	For any $\lambda\in \Lambda_\bbQ$ not on any hyperplanes $H_{\alpha^\vee, n}$, we define the set of objects $\{\fstab^{\bbZ}_{\lambda}(w)\mid w\in W\}$ in the derived category $D^b_{A_\bbZ\times(\Gm)_{\bbZ}}(T^*\calB_{\bbZ})$. 
	\begin{enumerate}
		\item Applying the derived base change functor $\otimes^L_{\bbZ}\bbC$ and then taking classes in the Grothendieck group, these objects become the $K$-theory stable basis for the alcove containing $\lambda$. 
		\item Applying the base change functor $\otimes^L_{\bbZ}k$ for an algebraically closed field $k$ of characteristic $p$ greater than the Coxeter number, up to $\lambda=-\frac{\lambda'+\rho}{p}$, these objects become the image of Verma modules over $U(\fg_k)$ with Harish-Chandra central character $\lambda'+2\rho$, under the localization functor $\gamma^{\lambda'}$ of \cite{BMR2} (recalled in \S~\ref{subsec:babyV}).
	\end{enumerate}
\end{thm}

In particular, under the duality between $A\times\Gm$-equivariant $K$-theory of $T^*\calB$ and the Iwahori-invariants of the principal series representation of the $p$-adic Langlands dual group, the standard objects, which are the Verma modules, are mapped to the standard objects, i.e.,  the characteristic functions. 
The restriction formula of \cite[Theorem~7.5]{SZZ} gives  a formula of the graded Ext between Verma modules and baby Verma modules with respect to the Koszul grading on \cite{BM}.

It has been well known that the $K$-theory lift of the monodromy of quantum connection of $T^*\calB$ defines an affine Hecke algebra action on the equivariant $K$-theory of $T^*\calB$ \cite{BMO}, which coincides with the one coming from \cite{BMR2}. An immediate consequence of Theorem~\ref{Thm:cat} is that the change of bases operator, induced by $K$-theory stable bases associated to different alcove, agrees with the monodromy operator in  $\pi_1(A^\vee_{reg}/W)$ of the quantum connection  of $T^*\calB$ (see Theorem \ref{thm:monodromystable} for a precise statement).

\subsection*{Organization of the paper}
In \S 2 we recall the definition and basic facts of stable bases, and in \S 3 we recall some basic facts of wall-crossings. In \S 4 we compute the formula of crossing walls determined by simple roots, which is generalized in \S 5 to non-simple walls. In \S 6 we recall  affine braid group actions on the derived categories of Springer resolutions, and use them to define categorified version of the stable basis. In \S~7, we show such basis coincides with Verma modules under the localization of Lie algebras in positive characteristic. As another consequence, we deduce in  \S~8 that wall-crossing of the stable basis agree with monodromy of quantum cohomology. In the Appendix we compute the wall crossing matrix in the case of $\SL(3,\bbC)$.

\section{Recollection of stable bases and the affine Hecke algebra}\label{sec:recoll}
\subsection{Notations}\label{sec:notation}
Let $G$ be a connected, semisimple, simply connected, complex linear algebraic group with a Borel subgroup $B$, whose roots are positive. Let $A\subset B$ be a maximal torus of $G$, $\fh$ be the Lie algebra of $A$, and $\Lambda$ (resp. $\Lambda^\vee$) the group of characters (resp. cocharacters) of $A$. Let $\rho$ denote the half sum of the positive roots. Let $\pm:=\fC_\pm\subset \Lambda^\vee\otimes_\bbZ\bbR$ denote the dominant/anti-dominant Weyl chambers in $\fh$. For any coroot $\alpha^\vee$, define the hyperplanes in $\mathfrak{h}^*_\bbR$ by \[H_{\alpha^\vee,n}=\{\lambda\in \mathfrak{h}_{\mathbb{R}}^*\mid(\lambda, \alpha^{\vee})=n\}, \quad n\in \bbZ.\]
We refer the connected components of 
\[\mathfrak{h}_{\mathbb{R}}^*\setminus (\underset{\al>0, n\in \bbZ}\cup H_{\alpha^\vee,n})\]
as the alcoves of this hyperplane arrangement. 
The walls of an alcove are the codimension-$1$ facets in the boundary of this alcove. 
The following alcove will be referred to as {\it the fundamental alcove}
\[\na_+:=\{\lambda\mid0<(\lambda, \alpha^\vee)<1, \text{ for any positive coroot } \alpha^\vee\}.\] 
In other words, it contains $\ep \rho$ for small $\ep>0$. Denote $\na_-=-\na_+. $ We say two alcoves $\na_1,\na_2$ are adjacent if they share a wall on some hyperplane $H_{\al^\vee, n}$.

We will use the same terminologies of alcoves, walls, and fundamental alcoves for their intersections with $\Lambda_\bbQ\subseteq\fh^*_{\bbR}$.

Denote by $\calB\simeq G/B$, the variety of all the Borel groups of $G$, and $T^*\calB$ (resp. $T\calB$) its cotangent bundle (resp. tangent bundle). For each $\la\in \Lambda$, there is an associated line bundle $\calL_\la$ on $\calB$ independent on the choice of $B$ \cite[\S6.1.11]{CG}. Pulling it back onto $T^*\calB$, we still denote it by $\calL_\la$. Therefore, we have identity 
\begin{equation}\label{eq:linebundle}
\Pic(T^*\calB)\otimes_\bbZ\bbR=\mathfrak{h}_{\mathbb{R}}^*.
\end{equation}

Let $\bbC^*$ act on $T^*\calB$ by $z\cdot(B', x)=(B', z^{-2}x)$, where $z\in \bbC^*$ and $(B', x)\in T^*\calB$. Let $q^{-1}$ denote the character of cotangent fiber under this action. We refer the readers to \cite{CG} for a beautiful account of the equivariant K-theory. Let $T=A\times \bbC^*$, and $K_{T}(\pt)\cong \bbZ[q^{1/2},q^{-1/2}][\Lambda]$. In this paper, we will consider $K_T(T^*\calB)$, which is a module over $K_T(\pt)$. The  $A$-fixed points of $T^*\calB$ are indexed by $W$;  each $w\in W$ determines the fixed point $wB\in \calB\subset T^*\calB$. We will just use $w$ to denote the corresponding fixed point $wB$. For each $w \in W$, denote $1_w\in K_{T}((T^*\calB)^A)\cong \oplus_{v\in W}K_T(\pt)$ the basis corresponding to $w$, and $\iota_w$ its image via the push-forward of  the embedding $i:w\to T^*\calB$, that is, $\iota_w=i_*(1)\in K_T(T^*\calB)$. It follows from the localization theorem \cite{CG} that $\{\iota_w|w\in W\}$ forms a basis for the localized equivariant K-theory $K_T(T^*\calB)_{\textit{loc}}:= K_T(T^*\calB)\otimes_{K_T(\pt)}\Frac K_T(\pt)$, where $\Frac K_T(\pt)$ denotes the fraction field of $K_T(\pt)$. For any vector space $V$ with a $T$-action, denote 
\[
\bigwedge {}^\bullet V=\sum_k(-1)^k\wedge^kV^\vee=\prod (1-e^{-\al})\in K_{T}(\pt),
\]
where the last product is over all the $T$-weights in $V$, counted with multiplicities. Note that this is not the standard notation for the wedge product because of the dual $^\vee$.

Recall there is a non-degenerate pairing $\langle\cdot, \cdot\rangle$ on  $K_T(T^*\calB)$ defined via localization:
	\begin{equation}\label{eqn:pairing}
	\langle\calF, \calG\rangle:=\sum_{w}\frac{\calF|_w\calG|_w}{\bigwedge {}^\bullet T_w(T^*\calB)}=\sum_{w}\frac{\calF|_w\calG|_w}{\prod_{\alpha>0}(1-e^{w\alpha})(1-qe^{-w\alpha})}\in \Frac K_T(pt),\end{equation}
	where $\calF, \calG\in  K_T(T^*\calB)$, and $\calF|_w$ denotes the pullback of $\calF$ to the fixed point $wB\in\calB\subset T^*\calB$.

\subsection{Definition of stable bases}

We recall the definition of stable bases of Maulik and Okounkov. 

Let $\fC$ be a chamber in $\Lambda^\vee_\bbR$. For any cocharacter $\sigma\in \fC$, the stable leaf (also called the attracting set) of the fixed point $w$ is 
\[
\Attr_\fC(w)=\{x\in T^*\calB\mid\underset{z\to 0}\lim \sigma(z)\cdot x=w\}.
\]
It  defines a partial order on $W$ as follows
\[
w\preceq_\fC v \quad \text{ if }\quad  \overline{\Attr_\fC(v)}\cap w\neq \emptyset.
\]
For example, with $\fC_+$  (resp. $\fC_-$) denoting the dominant chamber (resp. the anti-dominant chamber) where all the positive roots take positive (resp. negative) values, then 
\[
u\preceq_{\fC_+} v\iff u\le v, \quad \text{ and }\quad  u\preceq_{\fC_-}v \iff u\ge v,
\]
where $\le $ is the usual  Bruhat order of $W$. In this case, $\Attr_{\fC_+}(w)$ (resp. $\Attr_{\fC_-}(w)$) is equal to the conormal bundle of the Schubert cell $BwB/B$ (resp. opposite Schubert cell $B^-wB/B$, where $B^-$ is the opposite Borel subgroup) inside the flag variety $\calB$.
Denote the full attracting set
\[
\FAttr_\fC(v)=\bigcup_{w\preceq_{\fC} v}\Attr_\fC(w). 
\]

Define a polarization $T^{1/2}\in K_T(T^*\calB)$ to be an equivariant K-theory class such that the following identity holds,
\[
T^{1/2}+q^{-1}(T^{1/2})^\vee=T(T^*\calB)\in K_T(T^*\calB).
\]
Denote $T^{1/2}_{\opp}=q^{-1}(T^{1/2})^\vee$. We will mostly use the following two mutually opposite polarizations: $T\calB$ and $T^*\calB$, which lie in $K_{G\times \bbC^*}(T^*\calB)$.

Let $N_w:=T_w(T^*\calB)$ be the tangent space of $T^*\calB$ at the torus fixed point $wB$, and $T^{1/2}_w:=T^{1/2}|_w$. Each chamber $\fC$ determines a decomposition $N_w=N_{w,+}\oplus N_{w,-}$ of $N_w$ into $A$-weight spaces which are positive and negative with respect to $\fC$, and similarly $T^{1/2}_w=T^{1/2}_{w,+}\oplus T^{1/2}_{w,-}$. 
Then 
\[N_{w,-}=T^{1/2}_{w,-}\oplus q^{-1}(T^{1/2}_{w,+})^\vee,\]
and
\[N_{w,-}\ominus T^{1/2}_w=q^{-1}(T^{1/2}_{w,+})^\vee\ominus T^{1/2}_{w,+}.\]
Thus,
\[\left(\frac{\det N_{w,-}}{\det T^{1/2}_w}\right)^{1/2}=q^{-\frac{\rank T^{1/2}_{w,+}}{2}}\det (T^{1/2}_{w,+})^\vee\in K_T(\pt).\]

For any Laurent polynomial $f=\sum_{\mu\in \Lambda}f_\mu e^\mu\in K_T(\pt)$ with $f_\mu\in K_{\bbC^*}(\pt)$, define its {\it Newton polygon} to be 
\[
\deg_Af=\text{Convex hull }(\{\mu\mid f_\mu\neq 0\})\subset \Lambda\otimes_\bbZ\bbQ. 
\]

\begin{defn}\label{def:stable}\cite{OS, O} For any chamber $\fC$, polarization $T^{1/2}$,  alcove $\na$, there is a unique map  of $K_T(\pt)$-modules (called the stable envelope):
	\[
	\stab: K_{T}((T^*\calB)^{A})\to K_T(T^*\calB),
	\]
	satisfying the following three conditions. Denote $\stab^{\fC,T^{1/2}, \na}_w=\stab(1_w)$, then
	\begin{itemize}
		\item[(i)] (Support) $\supp( \stab^{\fC, T^{1/2},\na}_w)\subset \FAttr_{ \fC}(w)$.
		\item[(ii)] (Normalization) $\stab^{\fC,T^{1/2}, \na}_w|_w=(-1)^{\rank T^{1/2}_{w, +}}\left(\frac{\det N_{w,-}}{\det T_w^{1/2}}\right)^{1/2} \calO_{\Attr _\fC(w)}|_w.$
		\item[(iii)] (Degree) $\deg _A(\stab^{\fC, T^{1/2},\na}_w|_v)\subset \deg_A(\stab^{\fC,T^{1/2}, \na}_v|_v)+\calL|_v- \calL|_w$ for any $v\prec_\fC w, \calL\in \na$. 
	\end{itemize}
\end{defn}
\begin{rem}
	\begin{enumerate}
		\item[(a)] Using the identification in  \eqref{eq:linebundle}, $\calL\in\na$ is the same as a fractional line bundle. 
		\item[(b)] 
		The degree condition only depends on the choice of $\nabla$, not the fractional line bundle $\calL\in\na$ itself.  Moreover, the normalization condition does not depend on the alcove $\na$. 
		\item[(c)]
		We have duality \cite[Proposition 1]{OS}
		\begin{equation}
		\label{eq:dual}
		\langle \stab^{\fC, T^{1/2}, \na}_v, \stab^{-\fC, T^{1/2}_{\opp}, -\na}_w\rangle =\de_{v,w}. 
		\end{equation}
		\item[(d)]
		The existence of the K-theory stable bases also follows from the existence of the elliptic stable envelopes, see \cite{AO, O20a, O20b}.
	\end{enumerate}
\end{rem}
It follows immediately from the definition that we have the following lemma.
\begin{lem}\label{lem:dualequal} For any $w\in W$, we have
	\[
	\stab^{\fC, T^{1/2}, \na}_w|_w \cdot \stab^{-\fC, T^{1/2}_{\opp}, -\na}_w|_w=\bigwedge{}^\bullet T_w(T^*\calB). 
	\]
\end{lem}
We will mostly consider the following two cases
\[
\stp{\na}{w}:=\stab^{\fC_+, T\calB, \na}_w, \quad \stm{\na}{w}:=\stab^{\fC_-, T^*\calB, \na}_w. 
\]
From \cite[Lemma 3.2]{SZZ}, which follows immediately from Definition \ref{def:stable}(ii), we have
\begin{align}
\label{eq:normp}\stp{\na}{y}|_y&=q^{-\ell(y)/2}\prod_{\be>0, y\be<0}(q-e^{y\be})\prod_{\be>0,y\be>0}(1-e^{y\be}). \\
\label{eq:normm} \stm{\na}{y}|_y&=q^{\ell(y)/2}\prod_{\be>0, y\be<0}(1-e^{-y\be})\prod_{\be>0, y\be>0}(1-qe^{-y\be}). 
\end{align}
Moreover, it is easy to see from definition that 
\begin{equation}\label{eq:extremecases}
\stp{\na}{id}=[\calO_{T_{id}^*\calB}], \quad \textit{ and }\quad  \stm{\na}{w_0}=(-q^{1/2})^{\dim G/B}e^{2\rho}[\calO_{T_{w_0}^*\calB}],
\end{equation}
where $w_0\in W$ is the longest element in the Weyl group.

\subsection{The affine Hecke algebra}\label{sec:heckeaction}

We give a brief reminder on the Hecke algebra and Demazure-Lusztig operators.

Let $\bbH$ be the affine Hecke algebra, see \cite[\S 7.1]{CG}. Then we have the following well known Kazhdan-Lusztig and Ginzburg isomorphism \cite{KL,CG}
\begin{equation}\label{eq:iso}
\bbH \simeq K_{G\times \bbC^*}(Z),
\end{equation}
where $Z=T^*\calB\times_\calN T^*\calB$ is the Steinberg variety and $\calN\subset \frakg$ is the nilpotent cone. The right hand side $K_{G\times \bbC^*}(Z)$ has a convolution algebra structure. Let us recall one construction of this isomorphism \cite{R,L98}.

Let $\triangle(T^*\calB)\subset T^*\calB\times_{\calN} T^*\calB$ denote the diagonal copy of $T^*\calB$ inside the Steinberg variety $Z$. For any torus weight $\lambda$, let $\calO_\triangle(\lambda):=\calL_\lambda$, where $\calL_\lambda$ is the line bundle on $T^*\calB$ introduced in \S \ref{sec:notation}. For any simple root $\al$, let $P_\al$ be be the corresponding minimal parabolic subgroup and $\calP_\al=G/P_\al$. Let \[Y_\al=\calB\times_{\calP_\al}\calB\subset \calB\times\calB. \]
On $Y_\al$, the $A$-fixed points are of the form either $(w, w)$ or $(w, ws_\al)$ with $w\in W$. Let $T^*_{Y_\al}:=N^*_{\calB\times \calB/Y_\al}$ be the conormal bundle of $Y_\al$ inside $\calB\times\calB$. Let $\calO_{T_{Y_{\alpha}}^*}(\lambda,\mu)=\pi_1^*\calL_\lambda\otimes \pi_2^*\calL_\mu$, where $\pi_i:T^*_{Y_\al}\to  T^*\calB$ are the two projections. Let $T_\alpha$ be the usual generator of $\bbH$ that  satisfies $(T_\alpha+1)(T_\alpha-q)=0$. For the purpose of this paper, let us choose the isomorphism \eqref{eq:iso} as follows 
\[T_\alpha\mapsto -[\calO_{\bigtriangleup}]-[\calO_{T_{Y_{\alpha}}^*}(0,\alpha)],\text{\quad and \quad } e^\lambda\mapsto \calO_\triangle(\lambda).\]
This is conjugate to the one in \cite[Proposition 6.1.5]{R} by $\calL_\rho$, since our $q$ and  $T_\alpha$ are the same as $v^2$ and $vT_\alpha$ in \cite{R}, respectively. 

The convolution algebra $K_{G\times \bbC^*}(Z)$ is a subalgebra of $K_{A\times \bbC^*}(Z)$, and hence  defines two actions of $\bbH$ on $K_{A\times \bbC^*}(T^*\calB)$, namely, by convolution from the left, or from the right (which therefore is a right action). The operators on $K_{A\times \bbC^*}(T^*\calB)$  corresponding to these two convolutions with $T_\al$ will then be denoted by $T_\alpha$ and $T'_\alpha$ respectively,  following the notations used in \cite{SZZ}\footnote{When restricting our $T_\alpha$ to K-theory of $G/B$, our $T_\alpha^{-1}$ coincides with the $\calT_\al^\vee$ in \cite[p. 3]{AMSS19}}. That is, for any $\calF\in K_{A\times \bbC^*}(T^*\calB)$,
\[T_\alpha(\calF)=-\calF-\pi_{1*}(\pi_2^*\calF\otimes\pi_2^*\calL_\alpha),\textit{\quad and \quad} T'_\alpha(\calF)=-\calF-\pi_{2*}(\pi_1^*\calF\otimes \pi_2^*\calL_\alpha).\]
Let $D_\alpha:=-T_\alpha-1$ and $D'_\alpha:=-T'_\alpha-1$. Then we have adjointness (\cite[Lemma 4.4]{SZZ})
\begin{equation}
\label{eq:adj}
\langle D_\al(\calF), \calG\rangle =\langle \calF, D'_\al (\calG)\rangle, \quad \forall \calF,\calG\in K_{T}(T^*\calB). 
\end{equation}

The relation between these two operators $T_\al$ and $T'_\al$ can be deduced as follows. Since $[\calO_{T_{Y_\alpha}^*}(\rho,-\rho+\alpha)]=[\calO_{T_{Y_\alpha}^*}(-\rho+\alpha,\rho)]$ (see \cite[Lemma 1.5.1]{R}), the two operators on $K_{A\times \bbC^*}(T^*\calB)$, given by left and right convolutions with this sheaf, are equal to each other. In other words, for any $\calF\in K_{A\times \bbC^*}(T^*\calB)$,
\[
\pi_{1*}(\pi_2^*\calF\otimes[\calO_{T_{Y_\alpha}^*}(\rho,-\rho+\alpha)])=\pi_{2*}(\pi_1^*\calF\otimes[\calO_{T_{Y_\alpha}^*}(\rho,-\rho+\alpha)]).
\]
Therefore, we get the following equality as operators on $K_{A\times \bbC^*}(T^*\calB)$:
\[\calL_{\rho}T_\alpha\calL_{-\rho}=\calL_{-\rho} T'_\alpha\calL_{\rho}. \]
The operator of the left action  is denoted by $T^L_\alpha:=\calL_{\rho}T_\alpha\calL_{-\rho}$ and the right action is denoted by $T^{\oR}_\alpha:=\calL_{-\rho} T'_\alpha\calL_{\rho}$.

\begin{rem}\label{rmk:LR_Hecke}
	In the present paper, compositions of operators are read from right to left, even for right action operators. So for any $K$-theory class $\calF\in K_{A\times \bbC^*}(T^*(G/B))$,  $\calL_\rho(T_{\al}(\calL_{-\rho}\otimes \calF))= \calL_{-\rho}T'_{\al}(\calL_{\rho}\otimes \calF)$. For any reduced decomposition $w=s_1\cdots s_k$, we define $T_w^L:=T^L_{s_1}\cdots T^L_{s_k}$ and $T^{\oR}_w:=T^{\oR}_{s_k}\cdots T^{\oR}_{s_1}$.
	On the other hand, in  \cite{SZZ} a left action of the affine Hecke algebra through right convolution operators were used. In particular, under the notations of {\it loc. cit. } we have $T'_{s_1\cdots s_k}:=T'_{s_1}\cdots T'_{s_k}$. Compared to  the notations in the present paper we have, for any $x\in W$, \begin{equation}\label{eqn:T'T^R}T^{\oR}_x=\calL_{-\rho}T'_{x^{-1}}\calL_\rho.\end{equation}
\end{rem}

The following is one of the main results in \cite{SZZ}.
\begin{thm}\cite[Theorem 4.5]{SZZ} \label{lem:actiononstable}
	Let $\alpha$ be a simple root. Then
	\begin{align}
	\label{eq:Taction}T_\al(\stabm{w})&=\left\{ \begin{array}{ll}(q-1)\stabm{w}+q^{1/2}\stabm{ws_\al}, & \text{ if }ws_\al<w;\\
	q^{1/2}\stabm{ws_\al}, &\text{ if }ws_\al>w.\end{array}\right.\\
	T'_\al(\stabp{w})&=\left\{\begin{array}{ll}(q-1)\stabp{w}+q^{1/2}\stabp{ws_\al}, &\text{ if }ws_\al<w;\\
	q^{1/2}\stabp{ws_\al}, &\text{ if }ws_\al>w.\end{array}\right.
	\end{align}
	In particular,
	\[\stabm{w}=q^{\frac{\ell(w_0w)}{2}}T^{-1}_{w_0w}(\stabm{w_0}), \textit{\quad and \quad} \stabp{w}=q^{-\frac{\ell(w)}{2}}T'_{w^{-1}}(\stabp{id}).\]
\end{thm}

\section{General facts on the wall-crossing formulas}
The definition of a stable basis depends on a choice of a chamber and an alcove.
In this section, we collect some general facts about how the stable basis changes according to the change of these two choices involved. 

When the choice of the chamber is fixed,  the formula giving the change of bases associated to two different choices of alcove is called a wall-crossing formula.

\subsection{Duality of coherent sheaves}
Since $T^*\calB$ is smooth, $K_{T}(T^*\calB)$ is generated by vector bundles. For each vector bundle $\calF$, there is a dual vector bundle $\calF^\vee$. This operation is well-defined on $K_{T}(T^*\calB)$, and gives the duality operation
\[\calF\mapsto (\calF)^\vee,\] 
which sends $q$ to $q^{-1}$. Then we have the following relations:
\begin{lem}\label{lem:sheafdual}
	Let $w\in W$. Then
	\begin{align}
	\label{eq:dualm}(-q)^{\dim \calB}\calL_{-2\rho}\otimes(\stm{\na}{w})^\vee&=\stm{-\na}{w}\in K_{T}(T^*\calB),\\
	\label{eq:dualp}
	(-1)^{\dim \calB}\calL_{2\rho}\otimes(\stp{\na}{w})^\vee & =\stp{-\na}{w}\in K_{T}(T^*\calB).
	\end{align}
\end{lem}
\begin{rem}
See \cite[Lemma 11.1(b)]{AMSS19} or \cite[Equation (15)]{OS} for related statements. In our lemma, the polarizations on both sides of the equalities are unchanged, while those in \textit{loc. cit.} are changed to the opposite polarizations. 
\end{rem}
\begin{proof}
	We first show that \eqref{eq:dualm} implies \eqref{eq:dualp}. Indeed, by duality \eqref{eq:dual} and  \eqref{eq:dualm}, we have
	\begin{align*}
	\delta_{w,z}&=\langle (-q)^{\dim \calB}\calL_{-2\rho}\otimes(\stm{\na}{w})^\vee, \stp{\na}{z}\rangle\\
	&=(-q)^{\dim \calB}\sum_{y\in W}\frac{e^{-2y\rho}\ast(\stm{\na}{w}|_y)\stp{\na}{z}|_y}{\prod_{\alpha>0}(1-e^{y\alpha})(1-qe^{-y\alpha})},
	\end{align*}
	where $\ast$ in the second line acts on $K_T(pt)$ by sending $e^\lambda$ to $e^{-\lambda}$ and $q$ to $q^{-1}$. Applying $\ast$ to the above identity, we get
	\begin{align*}
	\delta_{w,z}&=(-q)^{-\dim \calB}\sum_{y\in W}\frac{e^{2y\rho}\stm{\na}{w}|_y\ast(\stp{\na}{z}|_y)}{\prod_{\alpha>0}(1-e^{-y\alpha})(1-q^{-1}e^{y\alpha})}\\
	&=(-1)^{\dim \calB}\sum_{y\in W}\frac{e^{2y\rho}\stm{\na}{w}|_y\ast(\stp{\na}{z}|_y)}{\prod_{\alpha>0}(1-e^{y\alpha})(1-qe^{-y\alpha})}\\
	&=\langle \stm{\na}{w}, (-1)^{\dim \calB}\calL_{2\rho}\otimes(\stp{\na}{z})^\vee\rangle.
	\end{align*}
Again by duality \eqref{eq:dual},	we then get \eqref{eq:dualp}.
	
	Now we prove \eqref{eq:dualm}. It suffices to show that $(-q)^{\dim \calB}\calL_{-2\rho}\otimes(\stm{\na}{w})^\vee$ satisfies the defining properties of $\stm{-\na}{w}$. The support condition is obvious. The normalization condition is checked as follows
	\begin{align*}
	\left((-q)^{\dim \calB}\calL_{-2\rho}\otimes(\stm{\na}{w})^\vee\right)|_w
	=&(-q)^{\dim \calB}e^{-2w\rho}\ast(\stm{\na}{w}|_w)\\
	\overset{\eqref{eq:normm}}=&(-q)^{\dim \calB}q^{-\frac{\ell(w)}{2}}e^{-2w\rho}\prod_{\beta>0, w\beta<0}(1-e^{w\beta})\prod_{\beta>0, w\beta>0}(1-q^{-1}e^{w\beta})\\
	=&q^{\frac{\ell(w)}{2}}\prod_{\beta>0, w\beta<0}(1-e^{-w\beta})\prod_{\beta>0, w\beta>0}(1-qe^{-w\beta})\\
	\overset{\eqref{eq:normm}}=&\stm{-\na}{w}|_w.
	\end{align*}
	Therefore, the normalization condition is verified.
	
	Finally, we check the degree condition. Pick $\calL\in -\na$. We need to check that for any $y> w$, the following is true
	\[\deg_A (-q)^{\dim \calB}\left(\calL_{-2\rho}\otimes(\stm{\na}{w})^\vee\right)|_y\subset \deg_A (-q)^{\dim \calB}\left(\calL_{-2\rho}\otimes(\stm{\na}{y})^\vee\right)|_y+\calL|_y-\calL|_w .\]
	Indeed, this in turn is equivalent to 
	\[\deg_A \stm{\na}{w}|_y\subset \deg_A \stm{\na}{y}|_y+\calL|_w-\calL|_y=\deg_A\stm{\na}{y}|_y+\calL^{-1}|_y-\calL^{-1}|_{w} .\]
	Since $\calL^{-1}\in \na$, the above identity is precisely the degree condition for the stable basis $\stm{\na}{w}$. Therefore, the degree condition also follows.  
\end{proof}

\subsection{Weyl group action}\label{subsec:WeylAction}
Recall on $K_T(T^*\calB)$, we have a left Weyl group action induced from the $G$-action on $T^*\calB$. Using equivariant localization, $w(\calF)$ for $\calF\in K_T(T^*\calB)$ and $w\in W$ is determined by its restrictions to the fixed points, which is given by
\begin{equation}
\label{eq:Wact}w(\calF)|_v=w(\calF|_{w^{-1}v}), \quad \calF\in K_{T}(T^*\calB).
\end{equation}
Moreover, the action by $w$ also acts on the base field $K_T(pt)$ by the usual Weyl group action. From the above formula, it is easy to get  
\[
w(\iota_{y})|_{v}=\left\{\begin{array}{cc} w(\iota_{y}|_y), & \text{if }v=wy,\\
0, & \text{otherwise}.\end{array} \right.
\]
That is, $w(\iota_{y})=\iota_{wy}$. This action agrees with the $\odot$-action defined in \cite[Definition 3.2]{LZZ}. The effect of this group action on a stable basis is the following, which can be proved in exactly the same way as the proof of Equation \eqref{eq:dualm} in Lemma \ref{lem:sheafdual} above. 
\begin{lem}\cite[Lemma 11.1(a)]{AMSS19}\label{lem:Weyl action}
	For any $w, y\in W$, we have
	\[w(\stab^{\mathfrak{C},T^{1/2},\na}_y)=\stab^{w\mathfrak{C},w(T^{1/2}),\na}_{wy}.\]
\end{lem}
\begin{rem}
\begin{enumerate} 
\item[(a)]
If the polarization $T^{1/2}=T\calB$ or $T^{1/2}=T^*\calB$, then $w(T^{1/2})=T^{1/2}$ as it lies in the $G$-equivariant K-theory, which are fixed by the left Weyl group action. 
\item[(b)]
A prior, the alcove on the right hand side should be $w(\na)$. I.e., for any $\lambda\in \na$, the fractional line bundle $\calL_\lambda$ should be changed to $w(\calL_\lambda)$. But since $\calL_\lambda$ is $G$-equivariant, $w(\calL_\lambda)=\calL_\lambda$. Thus, the alcove stays the same.
\end{enumerate}
\end{rem}

\subsection{Wall-crossing matrix}
Our computation of wall crossings for the K-theoretic stable bases is based on the following:
\begin{lem}\cite[Theorem 1]{OS}\label{lem:coefficients} 
	Suppose the alcoves $\na_1$ and $\na_2$ are adjacent and separated by $H_{\al^\vee,n}$. For any $y\in W$, we have
	\[
	\stab^{\fC, T^{1/2}, \na_1}_y=\left\{\begin{array}{ll}\stab^{\fC, T^{1/2}, \na_2}_y+f^{\na_1\leftarrow \na_2}_y\stab_{ys_\al}^{\fC, T^{1/2}, \na_2}, &\text{ if }ys_\al\prec_\fC y;\\
	\stab^{\fC, T^{1/2}, \na_2}_y , & \text{ if } ys_\al \succ_\fC y, \end{array}\right.
	\]
	where  $f_y^{\na_1\lefta\na_2}\in K_{T}(\pt)$. If $n=0$, $f_y^{\na_1\lefta \na_2}\in K_{\mathbb{C}^*}(\pt)$, i.e., it does not depend on the equivariant parameters of $A$. In particular, 
	\[f_y^{\na_1\lefta \na_2}=-f_y^{\na_2\lefta \na_1}.\]
\end{lem} 
\begin{proof} Let us first recall the moment map $\mu:(T^*\calB)^A\to H_2(T^*\calB, \bbZ)\otimes \Lambda$ defined in \cite[\S 2.1.7]{OS}. It is characterized by the following condition. If $p$ and $q$ are two fixed points connected by a torus invariant curve $C$, then 
	\[\mu(p)-\mu(q)=[C]\otimes \textit{weight}(T_pC).\]
This is also closely related to the GKM-condition in equivariant cohomology, see \cite[\S 4.8.5]{MO}.
	
	Identifying $H_2(T^*\calB, \bbZ)$ with the coroot lattice, we get
	\[\mu(y)-\mu(ys_\alpha)=\alpha^\vee\otimes(-y\alpha).\]
	For any other $z$ different from $y$ and $ys_\alpha$, $\mu(y)-\mu(z)$ will not be a multiple of $\alpha^\vee$.
	\footnote{The assumption that $\mu$ is positive in \textit{loc. cit.} is not necessary, since $n$ can be chosen as either positive or negative. See Equation (14) in \textit{loc. cit.}} The conclusion then follows from \textit{loc. cit.}, Theorem 1.
\end{proof}
We are interested in computing these wall crossing coefficients $f_y^{\na_1\lefta\na_2}$, which form the entries of the so-called \textit{wall R-matrix} (\cite[\S 2.2.3]{OS}). Similarly as stable bases, these matrices depend on the choice of the alcove, the chamber, and the polarization. For simplicity of notations, we omit the last two choices whenever they are clear from the context. 

From the definition, those wall crossing coefficients depend on the wall of the alcove lying on the hyperplane $H_{\al^\vee,n}$. However, it follows from Theorem \ref{thm:nonsimplewall} below, that the coefficients do not depend on the wall where we cross the hyperplane $H_{\al^\vee,n}$.

\begin{exmp}
We consider the case $\SL(2,\bbC)$ \footnote{See Appendix \ref{appendix} for the example of $\SL(3,\bbC)$, where wall-crossing coefficients for a non-simple root are computed.}. Then $\calB\cong\bbP^1$, and $\Pic(T^*\bbP^1)\otimes_\bbZ \bbQ=\bbQ\alpha$. The two $A$-fixed points on $T^*\bbP^1$ are denoted by $0$ and $\infty$, corresponding to $e $ and $s_\al$ in the Weyl group, respectively. The alcoves are $(\frac{n}{2}\alpha,\frac{n+1}{2}\alpha), n\in \bbZ$. The wall $H_{\alpha^\vee, 0}$ is  the origin $0$. We have $\na_+= (0,\frac{1}{2}\alpha)$ and $\na_-=(\frac{1}{2}\alpha, 0)=\na_+-\frac{\al}{2}$.  Using the translation formula of Lemma \ref{lem:shift}, we compute  $f_{s_\al}^{\na_-\lefta\na_+}$ for the dominant chamber and the polarization $T\bbP^1$, in which case $0\prec_{\fC_+}\infty$.  The stable basis is given by the following formulas (\cite[Example 2.4]{SZZ})
\begin{align*}
\stp{\na_-}{e}&=[\calO_{T_0^*\bbP^1}],\\
\stp{\na_-}{s_\al}&=-q^{-{1/2}}e^{-\alpha}[\calO_{\bbP^1}]+\left(-q^{{1/2}}e^{-2\alpha}+(q^{-{1/2}}-q^{{1/2}})e^{-\alpha}\right)[\calO_{T_0^*\bbP^1}].
\end{align*}
So their localizations are given by
\begin{align*}
\stp{\na_-}{e}|_e&=1-e^{\alpha}, &&\\
\stp{\na_-}{s_\al}|_e&=q^{{1/2}}-q^{-{1/2}}, & \stp{\na_-}{s_\al}|_{s_\al}&=q^{{1/2}}-q^{-{1/2}}e^{-\alpha}.
\end{align*}
By \eqref{eq:shift}, for any $y,w\in W$, 
\begin{align*}
\stp{\na_+}{y}|_w&=e^{-{1/2}y\alpha+{1/2}w\alpha}\stp{\na_-}{y}|_w.
\end{align*}
Therefore,
\begin{align*}
\stp{\na_+}{e}|_e&=1-e^{\alpha},&&\\
\stp{\na_+}{s_\al}|_e& =(q^{{1/2}}-q^{-{1/2}})e^\alpha, & \stp{\na_+}{s_\al}|_{s_\al}&=q^{{1/2}}-q^{-{1/2}}e^{-\alpha}.
\end{align*}
By Lemma \ref{lem:coefficients},
\[\stp{\na_-}{s_\al}=\stp{\na_+}{s_\al}+f_{s_\alpha}^{\na_-\leftarrow \na_+}\stp{\na_+}{e}.\]
Restricting both sides to the fixed point $e$, we get
\[q^{{1/2}}-q^{-{1/2}}=(q^{{1/2}}-q^{-{1/2}})e^\alpha+f_{s_\alpha}^{\na_-\leftarrow \na_+}(1-e^{\alpha}).\]
Hence
\[f_{s_\alpha}^{\na_-\leftarrow \na_+}=q^{{1/2}}-q^{-{1/2}}.\]
\end{exmp}

\subsection{Translations} We consider the effect of  translation by an integral weight $\mu\in \Lambda$ on a stable basis. By the uniqueness of stable basis, we have (see \cite[Lemma 8.2.(c)]{AMSS19}):
\begin{equation}\label{eq:shift}
\stab^{\fC,T^{1/2},\na+\mu}_y=e^{-y\mu}\calL_\mu\otimes \stab^{\mathfrak{C},T^{1/2},\na}_y.
\end{equation}
An immediate corollary of this fact is
\begin{lem}\label{lem:shift}Let $\na_1,\na_2$ are adjacent alcoves separated by $H_{\al^\vee, n}$. 
	For any integral weight $\mu\in \Lambda$, we denote $\na+\mu$ be the alcove obtained by translating $\na$ by $\mu$, then
	\[f_y^{\na_1+\mu\leftarrow \na_2+\mu}=e^{-(\mu,\alpha^\vee)y\alpha}f_y^{\na_1\leftarrow \na_2}.\]
\end{lem}
\begin{proof}
	By Lemma \ref{lem:coefficients} and Equation \eqref{eq:shift},
	\[f_y^{\na_1\leftarrow \na_2}=e^{y\mu-ys_\alpha\mu}f_y^{\na_1+\mu\leftarrow \na_2+\mu}.\]
	Hence the conclusion follows from the identity $\mu-s_\al\mu=(\mu, \al^\vee)\al$. 
\end{proof}

With this lemma, we can reduce the wall crossing for $H_{\alpha^\vee, n}$ to $H_{\alpha^\vee, 0}$ as follows.
Let $\na_1,\na_2$ be adjacent alcoves separated by $H_{\al^\vee, n}$.   If $\alpha$ is simple, then $\na_1-n\varpi_\alpha,\na_2-n\varpi_\alpha$ are adjacent alcoves separated by $H_{\al^\vee, 0}$, where $\varpi_\alpha$ is the corresponding fundamental weight. Otherwise, $\alpha=w\beta$ for some $w\in W$ and simple root $\beta$. Then we can shift the alcoves $\na_i$ by $-nw(\varpi_\beta)$ to get adjacent alcoves separated by $H_{\al^\vee, 0}$. 

Therefore, we only need to cross the walls on the root hyperplanes $H_{\alpha^\vee, 0}$. These are divided into two cases depending on whether $\alpha$ is simple or not.

\section{Crossing the simple walls}
In this section, we compute the formula for crossing the simple walls. 
\subsection{The formulas}
The following is the main result of this section.

\begin{thm}\label{thm:simplewall}
	Let $\alpha^\vee$ be a simple coroot. Suppose $\na$ has a wall on $H_{\al^\vee, 0}$, and $(\la, \al^\vee)>0, \forall \la\in \na$. Then
	\begin{align}
	\label{eq:heckewallp} \stp{s_\al\na}{y}=&
	\left\{ \begin{array}{ll}
	\stp{\na}{y}+(q^{{1/2}}-q^{-{1/2}})\cdot\stp{\na}{ys_{\alpha}},  & \textit{ if\quad } ys_{\alpha}< y;\\
	\stp{\na}{y},  & \text{ if\quad } ys_{\alpha}>y.
	\end{array}\right. \\ 
	\label{eq:heckewallm}
	\stm{s_\al\na}{y}=&
	\left\{ \begin{array}{ll}
	\stm{\na}{y}+(q^{{1/2}}-q^{-{1/2}})\cdot\stm{\na}{ys_{\alpha}}, & \textit{ if\quad } ys_{\alpha}> y;\\
	\stm{\na}{y}, & \text{ if\quad } ys_{\alpha}<y.
	\end{array}\right.
	\end{align}
\end{thm}
\begin{proof}
	We first show that Equation \eqref{eq:heckewallp} follows from Equation \eqref{eq:heckewallm} via duality \eqref{eq:dual}. By Lemma \ref{lem:coefficients}, we can assume $ys_\al<y$ and it suffices to compute
	\begin{align*}
	&\langle \stp{s_\al\na}{y}, \stm{-\na}{ys_{\alpha}}\rangle\\
	=&\langle \stp{s_\al\na}{y}, \stm{-s_\al\na}{ys_{\alpha}}+(q^{{1/2}}-q^{-{1/2}})\stm{-s_\al\na}{y}\rangle\\
	=&q^{{1/2}}-q^{-{1/2}},
	\end{align*}
	where the first equality follows from Equation \eqref{eq:heckewallm} with the alcove being $s_\alpha(-\na)$ satisfying the condition of  theorem.
	
	Now we prove Equation \eqref{eq:heckewallm}. From Lemma \ref{lem:coefficients} it suffices to assume that $ys_\al>y$. Applying the operator $D_\al$ from \S~\ref{sec:heckeaction} to the identity 
	\begin{equation*}
	\stm{\na}{y}+f^{s_\al\na\lefta\na}_y\stm{\na}{ys_\al}=\stm{s_\al \na}{y}, 
	\end{equation*}
	and use \eqref{eq:D.gen.p} and \eqref{eq:D.gen.m} below, we get 
	\begin{align*}
	&-\stm{\na}{y}-q^{1/2}\stm{\na}{ys_\al}+\sum_{w\not<ys_\al }a_w\stm{\na}{w}+f^{s_\al\na\lefta\na}_y(-q\stm{\na}{ys_\al}-q^{1/2}\stm{\na}{y}+\sum_{w\not<ys_\al }a_w\stm{\na}{w})\\
	&=-q\stm{s_\al\na}{y}-q^{1/2}\stm{s_\al\na}{ys_\al}+\sum_{w\not< ys_\al}a_w\stm{s_\al\na}{w}\\
	&\overset{\text{Lem.} \ref{lem:coefficients}}=-q(\stm{\na}{y}+f^{s_\al\na\lefta\na}_y\stm{\na}{ys_\al})-q^{1/2}\stm{\na}{ys_\al}+\sum_{w\not< ys_\al}a_w\stm{s_\al\na}{w},
	\end{align*}
	where the $a_w$ are some coefficients. These coefficients  are determined by \eqref{eq:D.gen.p} and \eqref{eq:D.gen.m} in Lemma~\ref{lem:generalD} below. 
	By comparing the coefficients of $\stm{\na}{y}$, we see that $f^{s_\al\na\lefta \na}_y=q^{1/2}-q^{-1/2}$. Note that the sums in the above identity will have no contribution to such coefficients. The first two are obvious since $ys_\alpha>y$. The reason for the last one is that $\stm{s_\al\na}w$ is a linear combination of $\stm{\na}{w}$ and $\stm{\na}{ws_\al}$, and both will not be equal to $\stm{\na}{y}$ when $w\not<ys_\al$. 
\end{proof}

\subsection{Some technical lemmas}
In this subsection we prove Lemma \ref{lem:generalD}, which is used in the proof of Theorem \ref{thm:simplewall}. First we recall a well-known fact.
\begin{lem}\cite[Theorem 1.1]{D}\label{lem:classical}
	Let $w_1, w_2\in W$, and $\alpha$ be a simple root. Assume $w_1s_\alpha<w_1$ and $w_2s_\alpha<w_2$. Then 
	\[w_1\leq  w_2\Leftrightarrow w_1s_\alpha\leq w_2\Leftrightarrow w_1s_\alpha\leq w_2s_\alpha.\]
\end{lem}

In preparation for the proof of Lemma  \ref{lem:generalD}, we have the following lemma. 
\begin{lem}\label{lem:wyz} Let $w, y\in W$. 
	\begin{enumerate}
		\item Assume    $y\le z\le w$. 
		\begin{enumerate}
			\item If $ys_\al<y$, $w\le y$, and  $w\not\in \{y, ys_\al\}$, then there is no such $z$.
			\item If $ys_\al>y$, $w\le ys_\al$, and $w\not\in \{y, ys_\al\}$, then there is no such $z$.
			\item If $w=y$, then $z=w=y$. 
			\item If $ys_\al<y$ and $w=ys_\al$, then there is no such $z$. 
			\item If $ys_\al>y$ and $w=ys_\al$, then $z=y$ or $z=ys_\al$.
		\end{enumerate}
		\item Assume $y\le z$ and $w\ge zs_\al$. 
		\begin{enumerate}
			\item If $ys_\al<y$,  $w\le y$, and $w\not\in \{y, ys_\al\}$, then there is no such $z$.
			\item If $ys_\al>y$, $w\le ys_\al$, and $w\not\in \{y, ys_\al\}$, then there is no such $z$.
			\item If $ys_\al<y$ and $w=y$, then $z=y$. 
			\item If $ys_\al>y$ and $w=y$, then $z=ys_\al$. 
			\item If $ys_\al<y$ and $w=ys_\al$, then $z=y$. 
			\item If $ys_\al>y$ and $w=ys_\al$, then $z=y$ or $z=ys_\al$. 
		\end{enumerate}
		\item If $w>\min\{y, ys_\al\}$ and $w\neq y, ys_\al$, then $y<\max\{w, ws_\al\}$.
	\end{enumerate}
\end{lem}
\begin{proof}
	(1a)-(1d) are obvious. For (1e), we have $ys_\al=w\ge z\ge y$. So $z=y$ or $z=ys_\al$. 
	
	(2a): $w<y$ and $w\neq ys_\al$ imply that $
	z\ge y>w\ge zs_\al$, 
	so
	\[
	\ell(z)\ge \ell(y)>\ell(w)\ge \ell(zs_\al)\ge \ell(z)-1.
	\]
	So $\ell(z)=\ell(y)$ and $\ell(w)=\ell(zs_\al)$. That is, $z=y$ and $w=zs_\al$. Hence, $w=ys_\al$, which is  a contradiction. So no such $z$ exists. 
	
	(2b): If $z=y$, then $w\ge ys_\al$, which contradicts to $ys_\alpha>w$. So $y<z$.
	
	If $zs_\al>z$, then $\ell(y)+1=\ell(ys_\al)> \ell(w)\geq \ell(zs_\al)=\ell(z)+1$. So $\ell(y)>\ell(z)$. This contradicts to the assumption that $y\le z$. 
	
	If $zs_\al<z$, then $\ell(y)+1=\ell(ys_\al)>\ell(w)\geq \ell(zs_\al)=\ell(z)-1$. So $\ell(y)+2>\ell(z)> \ell(y)$ which forces $\ell(z)=\ell(y)+1$, i.e., $z=ys_\be$ for some positive root $\be$ such that $y\be>0$.  We claim $\al=\be$. If not, then $s_\al\be>0$.  But $\ell(ys_\al)=\ell(y)+1=\ell(z)>\ell(zs_\al)=\ell(ys_\be s_\al)=\ell(ys_\al s_{s_\al \be})$, which implies $ys_\al(s_\al(\be))=y\beta<0$. This contradicts to the assumption that $y\be>0$.  Therefore $\al=\be$, and $z=ys_\al$. 
	
	Then $ys_\al>w\geq zs_\al=y$. Since $w\neq y$, such $w$ does not exist.
	
	(2c): We have $z\ge y\ge zs_\al$. So $z=y$ or $z=ys_\al$. Since $z\geq y$, $z$ must be equal to $y$.
	
	(2d): The proof is similar as that of (2c). 
	
	(2e): Since $zs_\al\le ys_\al<y\le z$, $z=y$. 
	
	(2f): It is easy to see that $z=y$ and $z=ys_\al$ satisfy the conditions. Now assume that $ys_\al>y, z>y$ and $ys_\al>zs_\al$. Then $\ell(y)+1=\ell(ys_\al)>\ell(zs_\al)\ge \ell(z)-1$. Hence, $\ell(y)+2>\ell(z)>\ell(y)$, which implies that $\ell(z)=\ell(y)+1$. It follows similarly as the proof of (2b) that $z=ys_\al$. 
	
	(3): It follows from a case-by-case study, using Lemma~\ref{lem:classical}.
\end{proof}

Now we are ready to prove the following technical lemma, which has been used in the proof of Theorem \ref{thm:simplewall}. 
\begin{lem}\label{lem:generalD} Let $\al$ be a simple root, and $\na$ be an alcove whose closure intersects with $H_{\al^\vee, 0}$ non-trivially. 
	\begin{enumerate}
		\item If  $(\la, \al^\vee)>0, \forall\la\in \na$, then 
		\begin{align}\label{eq:D.gen.p}
		D_\alpha(\stm{\na}{y})&=\left\{\begin{array}{cc}-q\stm{\na}{y}
		-q^{{1/2}}\stm{\na}{ys_{\alpha}}+\sum_{w\not< y}a_w\stm{\na}{w}, & \text{ if }ys_\al<y;\\
		-\stm{\na}{y}-q^{1/2}\stm{\na}{ys_\al}+\sum_{w\not< ys_\al}a_w \stm{\na}{w}, & \text{ if } ys_\al>y.\end{array}\right.
		\end{align}
		\item If $(\la, \al^\vee)<0, \forall \la\in \na$, then 
		\begin{align}\label{eq:D.gen.m}
		D_\alpha(\stm{\na}{y})&=\left\{\begin{array}{cc}-\stm{\na}{y}
		-q^{{1/2}}\stm{\na}{ys_{\alpha}}+\sum_{w\not< y}a_w\stm{\na}{w}, & \text{ if }ys_\al<y;\\
		-q\stm{\na}{y}-q^{1/2}\stm{\na}{ys_\al}+\sum_{w\not< ys_\al}a_w \stm{\na}{w}, & \text{ if } ys_\al>y.\end{array}\right.
		\end{align}
		Here $a_w\in K_T(\pt)$. Note that we have been sloppy about the coefficients $a_w$, since they clearly depend on $\na$, $y$, and $\al$. 
	\end{enumerate}
\end{lem}
\begin{proof}
	The proof is similar to that for \cite[Proposition 4.3]{SZZ}.\\
	{\it Step 1}. 
	By duality \eqref{eq:dual}, 
	\begin{align*}
	D_\alpha(\stm{\na}{y})
	=\sum_{w\in W}a_w\stm{\na}{w}, \text{with }a_w=\langle D_\alpha(\stm{\na}{y}), \stp{-\na}{w}\rangle. 
	\end{align*}
	By the support condition of the stable basis, $a_w=\langle D_\alpha(\stm{\na}{y}), \stp{-\na}{w}\rangle$ is a proper integral. Hence $a_w\in  K_T(\pt)$. 
	By localization formula and the support condition (see the first several lines in the proof of \cite[Proposition 4.3]{SZZ}), 
	\begin{align}\label{eq:ay}
	a_w=\sum_{y\leq z\leq w}\frac{\stp{-\na}{w}|_z\stm{\na}{y}|_z}{\bigwedge^\bullet T_z(T^*\calB)}\frac{e^{z\alpha}-q}{1-e^{z\alpha}}
	+\sum_{w\geq zs_\alpha\text{ and } y\leq z}\frac{\stp{-\na}{w}|_{zs_\alpha}\stm{\na}{y}|_z}{\bigwedge^\bullet T_z(T^*\calB)}\frac{1-qe^{-z\alpha}}{1-e^{-z\alpha}}e^{z\alpha}.
	\end{align}
	
	From Lemma \ref{lem:wyz} above, we know that if $ys_\al<y$, $w<y$ and $w\not\in \{y, ys_\al\}$, or if $ys_\al>y$, $w<ys_\al$ and $w\not\in \{y, ys_\al\}$, then the two sums are empty, so $a_w=0$. Therefore, we just need to consider $a_y$ and $a_{ys_\al}$. 
	
	In the following Steps 2-5, we assume $ys_\al<y$, and starting from Step 6, we assume $ys_\al>y$. 
	
	\textit{Step 2}. We compute $a_{ys_\al}$, assuming   $ys_\al<y$. By Lemma \ref{lem:wyz},  there is only one  term in the second sum with $z=y$, then 
	\begin{align*}a_{ys_\alpha}&=\frac{\stp{-\na}{ys_\alpha}|_{ys_\alpha}\stm{\na}{y}|_y}{\bigwedge^\bullet T_y(T^*\calB)}\frac{1-qe^{-y\alpha}}{1-e^{-y\alpha}}e^{y\alpha}\\
	&\overset{\text{Lem. \ref{lem:dualequal}}}=\frac{\stp{-\na}{ys_\al}|_{ys_\al}}{\stp{-\na}{y}|_y}\frac{1-qe^{-y\al}}{1-e^{-y\al}}e^{y\al}\\
	&\overset{\eqref{eq:normp}}=\frac{q^{-\ell(ys_\al)/2}\prod_{\be>0, ys_\al\be<0}(q-e^{ys_\al\be})\prod_{\be>0, ys_\al\be>0}(1-e^{ys_\al \be})}{q^{-\ell(y)/2}\prod_{\be>0, y\be<0}(q-e^{y\be})\prod_{\be>0, y\be>0}(1-e^{y\be})} \frac{1-qe^{-y\al}}{1-e^{-y\al}}e^{y\al}\\
	&\overset{\sharp}=q^{1/2}\frac{1-e^{-y\al}}{q-e^{y\al}}\frac{1-qe^{-y\al}}{1-e^{-y\al}}e^{y\al}
	\\
	&=-q^{{1/2}},
	\end{align*}
	where  $\sharp$ follows from the facts $y\alpha<0$ and $s_\alpha\{\beta>0|\beta\neq \alpha\}=\{\beta>0|\beta\neq \alpha\}$,  since $\alpha$ is simple.

	{\it Step 3}. We consider $a_y$, assuming $ys_\al<y$. By Lemma \ref{lem:wyz} above, there is only one term in each sum with both $z=y$, and we have
	\begin{align*}
	a_y&=\frac{\stp{-\na}{y}|_y\stm{\na}{y}|_y}{\bigwedge^\bullet T_y(T^*\calB)}\frac{e^{y\alpha}-q}{1-e^{y\alpha}}
	+\frac{\stp{-\na}{y}|_{ys_\alpha}\stm{\na}{y}|_y}{\bigwedge^\bullet T_y(T^*\calB)}\frac{1-qe^{-y\alpha}}{1-e^{-y\alpha}}e^{y\alpha}\\
	&\overset{\text{Lem. \ref{lem:dualequal}}}=\frac{e^{y\alpha}-q}{1-e^{y\alpha}}+\frac{\stp{-\na}{y}|_{ys_\alpha}\stm{\na}{y}|_y}{\bigwedge^\bullet T_y(T^*\calB)}\frac{1-qe^{-y\alpha}}{1-e^{-y\alpha}}e^{y\alpha}\\
	&=\frac{e^{y\al}-q}{1-e^{y\al}}(1-\frac{\stp{-\na}{y}|_{ys_\alpha}\stm{\na}{y}|_y\cdot e^{y\al}}{\bigwedge^\bullet T_y(T^*\calB)})
	\end{align*}

	{\it Step 4}.  In this step we compute $a_y$, assuming $(\la, \al^\vee)>0, \la\in \na$. For any Laurent polynomial $g=\sum_{\mu\in I}c_\mu e^\mu$ and $\xi\in \fh$, denote
	\[
	g(\xi)=\sum_{\mu\in I}c_\mu e^{(\mu, \xi)}, \quad 
	\Max_\xi(g)=\max_{\mu\in I}\{(\mu, \xi)\}, \quad \Min_\xi(g)=\min_{\mu\in I}\{(\mu, \xi)\}. 
	\] 
	For example, for any $y\in W$, 
	\begin{equation}\label{eq:rhopair} 
	\Max_\xi(\bigwedge\nolimits^\bullet T_y(T^*\calB))=-\Min_\xi(\bigwedge\nolimits^\bullet T_y(T^*\calB))=\left\{\begin{array}{cc}(2\rho, \xi), & \text{ if }\xi\in \fC_+;\\
	-(2\rho, \xi), & \text{ if }\xi\in \fC_-. \end{array}\right.
	\end{equation}
	
	Denote
	\[
	f=\stp{-\na}{y}|_{ys_\al} \stm{\na}{y}|_y \cdot e^{y\al}, \text{ and } K_T(pt)\ni a_y=\sum_{\mu\in I}c_\mu e^\mu, c_\mu\in \bbZ[q^{1/2}, q^{-1/2}].\]
	Here $I\subset X^*(T)$ is a finite subset. Following  \cite[Lemma 3.1]{SZZ}, let $\xi$ be in $\fC_+$ (so that $(2\rho, \xi)>0$) such that 
	\begin{equation}\label{eq:xi}(\mu, \xi)\in \bbZ\backslash\{0\}, \quad (\mu, \xi)\neq (\mu', \xi), \quad \forall\mu\neq  \mu'\in I. 
	\end{equation}
	Since $ys_\al<y$, $y\al<0$, so $(y\al, \xi)<0$. Since $(\la, \al^\vee)>0$, we adjust $\la\in \na$ close to $H_{\al^\vee, 0}$ so that 
	\[
	-1<(y\la-ys_\al\la,  \xi)=(\la, \al^\vee)(y\al, \xi)<0. 
	\]
	The degree condition together with \cite[(9)]{SZZ}  imply that 
	\begin{align*}
	\Max_\xi(f)&\le \Max_\xi(\stp{-\na}{ys_\al}|_{ys_\al})+( \calL_{-\la}|_{ys_\al}-\calL_{-\la}|_y, \xi)+\Max_\xi(\stm{\na}{y}|_y)+(y\al, \xi)\\
	&=(2\rho, \xi)+(\la, \al^\vee)(y\al, \xi). 
	\end{align*}
	Since $\Max_\xi(f)\in \bbZ$,   $\Max_\xi(f)\le (2\rho, \xi)-1$. 
	Therefore, by using \eqref{eq:rhopair}, we have
	\[\underset{t\to \infty}\lim \frac{f(t\xi)}{\bigwedge^\bullet T_y(T^*\calB)(t\xi)}=0. \] So 
	\[
	\underset{t\to \infty}\lim a_y(t\xi)= \underset{t\to \infty}\lim \left(\frac{e^{(y\al, t\xi)}-q}{1-e^{(y\al, t\xi)}}(1-\frac{f(t\xi)}{\bigwedge^\bullet T_y(T^*\calB)(t\xi)})\right)=\frac{0-q}{1-0}\cdot (1-0)=-q. 
	\]
	
	Similarly, we have \cite[(10)]{SZZ}
	\begin{align*}
	\Min _\xi(f)&\ge \Min_\xi(\stp{-\na}{ys_\al}|_{ys_\al})+(\calL_{-\la}|_{ys_\al}-\calL_{-\la}|_{y}, \xi)+\Min_\xi(\stm{\na}{y}|_y)+(y\al, \xi)\\
	&=-(2\rho, \xi)+(\la, \al^\vee)(y\al, \xi).
	\end{align*}
	Since $\Min_\xi(f_1)\in \bbZ$, so $\Min_\xi(f_1)\ge -(2\rho, \xi)$. From \eqref{eq:rhopair}, we get
	\[
	\underset{t\to -\infty}\lim\frac{f(t\xi)}{\bigwedge^\bullet T_y(T^*\calB)(t\xi)} \text{ is bounded}. 
	\]
	Therefore, 
	\[
	\underset{t\to -\infty}\lim a_y(t\xi)= \underset{t\to -\infty}\lim\left(\frac{e^{(y\al, t\xi)}-q}{1-e^{(y\al, t\xi)}}(1-\frac{f(t\xi)}{\bigwedge^\bullet T_y(T^*\calB)(t\xi)})\right), 
	\]
	which is bounded. Since $a_y\in K_T(\pt)$,  $a_y(t\xi)$ is a Laurent polynomial in variable $e^t$. The above two limits show that  $a_y=-q$  is a constant Laurent polynomial. 
	
	{\it Step 5}. In this step we compute $a_y$ assuming   $(\la, \al^\vee)<0, \la\in \na$. We can pick $\xi\in \fC_-$ (so that $(2\rho, \xi)<0$) satisfying  \eqref{eq:xi}. In this case,  $(y\al, \xi)>0$, and we adjust $\la\in \na$ close to $H_{\al^\vee, 0}$ so that 
	\[
	-1<(y\la-ys_\al\la, \xi)=(\la, \al^\vee)(y\al, \xi)<0. 
	\]
	Then similarly as in Step 4, we can show that
	\[
	\underset{t\to \infty}\lim a_y(t\xi)=\underset{t\to \infty}\lim \left(\frac{e^{(y\al, t\xi)}-q}{1-e^{(y\al, t\xi)}}(1-\frac{f(t\xi)}{\bigwedge^\bullet T_y(T^*\calB)(t\xi)})\right)=-1,
	\]
	and  $\underset{t\to -\infty}\lim a_y(t\xi)$ is bounded. Hence, $a_y=-1$.

	{\it Step 6}. We now consider the case  $ys_\al>y$. 
	For $a_y$ with $ys_\al>y$, then in \eqref{eq:ay}, $z=y$ in the first sum and $z=ys_\al$ in the second, so 
	\begin{align*}
	a_y&=\frac{\stp{-\na}{y}|_y\stm{\na}{y}|_y}{\bigwedge^\bullet T_y(T^*\calB)}\frac{e^{y\al}-q}{1-e^{y\al}}+\frac{\stp{-\na}{y}|_{y}\stm{\na}{y}|_{ys_\al}}{\bigwedge^\bullet T_{ys_\al}(T^*\calB)}\frac{1-qe^{y\al}}{1-e^{y\al}}e^{-y\al}\\
	&=\frac{e^{y\al}-q}{1-e^{y\al}}+\frac{\stp{-\na}{y}|_{y}\stm{\na}{y}|_{ys_\al}}{\bigwedge^\bullet T_{ys_\al}(T^*\calB)}\frac{1-qe^{y\al}}{1-e^{y\al}}e^{-y\al}. 
	\end{align*}
	Now let $k=\stp{-\na}{y}|_{y}\stm{\na}{y}|_{ys_\al}e^{-y\al}$. 
	
	{\it Step 7}. Suppose $(\la, \al^\vee)>0, \forall \la\in \na$. We choose $\xi\in \fC_+$ (so that $(2\rho, \xi)>0$ and $(y\al, \xi)>0$) satisfying properties similar as \eqref{eq:xi}. We pick $\la\in \na$,  so that 
	\[
	-1<-(\la, \al^\vee)(y\al, \xi)<0.
	\]
	Then 
	\[
	\Max_\xi(k)\le \Max_\xi(\stp{-\na}{y}|_y)+\Max_\xi(\stm{\na}{ys_\al}|_{ys_\al})+(\calL_\la|_{ys_\al}-\calL_\la|_y, \xi)-(y\al, \xi)=(2\rho, \xi)-(\la, \al^\vee)(y\al, \xi). 
	\]
	So 
	\[
	\underset{t\to \infty}\lim a_y(t\xi)=-1. 
	\]
	On the other hand, we have
	\[
	\Min_\xi(k)\ge \Min_\xi(\stp{-\na}{y}|_y)+\Min_\xi(\stm{\na}{ys_\al}|_{ys_\al})+(\calL_\la|_{ys_\al}-\calL_\la|_y, \xi)-(y\al, \xi)\ge (-2\rho, \xi)-(\la, \al^\vee)(y\al, \xi).
	\]
	Since $\Min_\xi(k)\in \bbZ$, so $\Min_\xi(k)\ge (-2\rho, \xi)$, therefore, 
	\[
	\underset{t\to -\infty}\lim a_y(t\xi) \text{ is bounded}. 
	\]
	So $a_y=-1$. 
	
	\textit{Step 8}. If $(\la, \al^\vee)<0, \forall \la\in \na$, then we can mimic Step 7 by picking $\xi\in \fC_-$ (so that $(2\rho, \xi)<0$ and $(y\al, \xi)<0$), and pick $\la\in \na$ so that $-1<-(\la, \al^\vee)(y\al, \xi)<0$, then 
	\[
	\underset{t\to \infty}\lim a_y(t\xi)=-q, \quad \text{ and }\underset{t\to -\infty}\lim a_y(t\xi) \text{ is bounded. }
	\]
	So $a_y=-q.$
	
	\textit{Step 9.}
	We then compute $a_{ys_\al}$ assuming $ys_\al>y$. In \eqref{eq:ay}, $z=y$ and $z=ys_\al$ in both sums, so we have 
	\begin{align*}
	a_{ys_\al}&=\frac{\stp{-\na}{ys_\al}|_y\stm{\na}{y}|_y}{\bigwedge^\bullet T_y(T^*\calB)}\frac{e^{y\al}-q}{1-e^{y\al}}+\frac{\stp{-\na}{ys_\al}|_{ys_\al}\stm{\na}{y}|_{ys_\al}}{\bigwedge^\bullet T_{ys_\al}(T^*\calB)}\frac{e^{-y\al}-q}{1-e^{-y\al}}\\
	&+\frac{\stp{-\na}{ys_\al}|_{ys_\al}\stm{\na}{y}|_{y}}{\bigwedge^\bullet T_{y}(T^*\calB)}\frac{1-qe^{-y\al}}{1-e^{-y\al}}e^{y\al}+\frac{\stp{-\na}{ys_\al}|_{y}\stm{\na}{y}|_{ys_\al}}{\bigwedge^\bullet T_{ys_\al}T^*\calB}\frac{1-qe^{y\al}}{1-e^{y\al}}e^{-y\al}. 
	\end{align*}
	Denote 
	\begin{align*}
	f_1=\stp{-\na}{ys_\al}|_y\stm{\na}{y}|_y, \quad f_2=\stp{-\na}{ys_\al}|_{ys_\al}\stm{\na}{y}|_{ys_\al}, \\
	f_3=\stp{-\na}{ys_\al}|_{ys_\al}\stm{\na}{y}|_{y}e^{y\al}, \quad f_4=\stp{-\na}{ys_\al}|_{y}\stm{\na}{y}|_{ys_\al}e^{-y\al}. 
	\end{align*}
	\textit{Step 10}. If $(\la, \al^\vee)>0, \forall\la\in \na$, then we can choose $\xi\in \fC_+$ satisfying  \eqref{eq:xi} and  $(y\al, \xi)>0$, and pick $\la\in \na$ satisfying $-1<-2(\la, \al^\vee)(y\al, \xi)<0$, then 
	\[
	\underset{t\to \infty}\lim f_1(t\xi)=\underset{t\to \infty}\lim f_2(t\xi)=\underset{t\to \infty}\lim f_4(t\xi)=0,
	\]
	and it is easy to calculate from Equations \eqref{eq:normm} and \eqref{eq:normp} that
	\[
	\frac{f_3}{\bigwedge^\bullet T_y(T^*\calB)}=q^{-1/2}\frac{q-e^{-y\al}}{1-e^{y\al}}. 
	\]
	Therefore, 
	\[
	\underset{t\to \infty}\lim a_{ys_\al}(t\xi)=-q^{1/2}, \quad \underset{t\to -\infty}\lim a_{ys_\al}(t\xi) \text{ is bounded. }
	\]
	
	Similarly, if $(\la, \al^\vee)<0, \forall \la\in \na$, then we can choose $\xi\in \fC_-$  (so that $(y\al, \xi)<0$), and $\la\in \na$ satisfying  $-1<-2(\la, \al^\vee)(y\al, \xi)<0$, then 
	\[
	\underset{t\to \infty}\lim a_{ys_\al}(t\xi)=-q^{1/2}, \quad \underset{t\to -\infty}\lim a_{ys_\al}(t\xi) \text{ is bounded. }
	\]
	Therefore, $a_{ys_\al}=-q^{1/2}$. 
\end{proof}

\subsection{Arbitrary Chamber}

We generalize Theorem \ref{thm:simplewall} to the setting of an arbitrary chamber. 
The choice of a chamber $\fC$ determines a set of simple roots. Replacing the positive chamber $+$ in Theorem \ref{thm:simplewall} by the chamber $\fC$, we get the following result.

\begin{thm}\label{thm:simplewall.gen}
	Let $\alpha^\vee$ be a simple coroot determined by the chamber $\fC$. Suppose $\na$ has a wall on $H_{\al^\vee, 0}$, and $(\la, \al^\vee)>0, \forall \la\in \na$. Then
	\begin{align*}
	\stab^{\fC, T\calB, s_\al\na}_{y}=&
	\left\{ \begin{array}{ll}
	\stab^{\fC, T\calB, \na}_{y}+(q^{{1/2}}-q^{-{1/2}})\cdot\stab^{\fC, T\calB, \na}_{ys_{\alpha}},  & \textit{ if\quad } ys_{\alpha}\prec_\fC y;\\
	\stab^{\fC, T\calB,\na}_{y},  & \text{ if\quad } ys_{\alpha}\succ_\fC y.
	\end{array}\right. \\
	\stab^{-\fC, T^*\calB, s_\al\na}_{y}=&
	\left\{ \begin{array}{ll}
	\stab^{-\fC, T^*\calB, \na}_{y}+(q^{{1/2}}-q^{-{1/2}})\cdot\stab^{-\fC, T^*\calB, \na}_{ys_{\alpha}}, & \textit{ if\quad } ys_{\alpha}\succ_\fC y;\\
	\stab^{-\fC, T^*\calB, \na}_{y}, & \text{ if\quad } ys_{\alpha}\prec_\fC y.
	\end{array}\right.
	\end{align*}
\end{thm}
Thanks to this result, in the rest of the paper, we focus on stable bases associated to the positive/negative chamber.

\section{Wall-crossings and the affine Hecke algebra action}
In this section we compute the formula when crossing a wall of a non-simple root, and also study the relation between wall crossing and the affine Hecke algebra action.
\subsection{Crossing the non-simple walls}\label{sec:non-simple}

\begin{thm}\label{thm:nonsimplewall}
	Suppose $\nabla$ has a wall on $H_{\alpha^\vee, 0}$ (so that $\na$ and $s_\al\na$ share a wall on  $H_{\alpha^\vee, 0}$), and    $(\lambda,\alpha^\vee)>0, \forall \la\in \na$.   
	Then 
	\begin{align*}
	\stab^{\fC, T\calB,s_\al\na}_{y}=&
	\left\{ \begin{array}{ll}
	\stab^{\fC, T\calB,\na}_{y}+(q^{{1/2}}-q^{-{1/2}})\cdot\stab^{\fC, T\calB, \na}_{ys_{\alpha}},  & \textit{ if\quad } ys_{\alpha}\prec_\fC y;\\
	\stab^{\fC, T\calB, \na}_{y},  & \text{ if\quad } ys_{\alpha}\succ_\fC y.
	\end{array}\right. \\ 
	\stab^{-\fC, T^*\calB, s_\al\na}_{y}=&
	\left\{ \begin{array}{ll}
	\stab^{-\fC, T^*\calB, \na}_{y}+(q^{{1/2}}-q^{-{1/2}})\cdot\stab^{-\fC, T^*\calB, \na}_{ys_{\alpha}}, & \textit{ if\quad } ys_{\alpha}\succ_\fC y;\\
	\stab^{-\fC, T^*\calB, \na}_{y}, & \text{ if\quad } ys_{\alpha}\prec_\fC y.
	\end{array}\right.
	\end{align*}
\end{thm}


\begin{rem}
	This theorem shows that the wall crossing coefficients $f_y^{s_\al\na\lefta\na}$ does not depend on where we cross the root hyperplane $H_{\alpha^\vee, 0}$, which is not obvious from the definitions. In other words, for any $\na$ having a wall on $H_{\al^\vee,0}$, the coefficient $f_y^{s_\al\na\lefta\na}$ does not depend on the choice of $\na$.
\end{rem}

\begin{proof}
	It suffices to prove the first one, since the other one will follow from it via duality as in Theorem \ref{thm:simplewall}.
	Assume first that $\fC=\fC_+$, then we just need to consider the case when $ys_\al<y$ due to Lemma \ref{lem:coefficients}. We have
	\[\stab^{\fC_+, T\calB, s_\al\na}_{y}=\stab^{\fC_+, T\calB, \na}_{y}+f_y^{s_\al\na\lefta\na}\cdot\stab^{\fC_+, T\calB, \na}_{ys_{\alpha}},\]
	where $f^{s_\al\na\lefta\na}\in K_{\bbC^*}(pt)$ due to Lemma \ref{lem:coefficients}.
	Suppose $\alpha^\vee=w\beta^\vee$ for some simple coroot $\beta^\vee$ and $w\in W$. Applying $w$ to the above identity and using Lemma \ref{lem:Weyl action}, we get
	\[\stab^{w(\fC_+),T\calB, s_\al\na}_{wy}=\stab^{w(\fC_+),T\calB, \na}_{wy}+f_y^{s_\al\na\lefta\na}\cdot\stab^{w(\fC_+),T\calB, \na}_{wys_{\alpha}}.\]
	Here we have used the fact that $w(T\calB)=T\calB$ and $w(f_y^{s_\al\na\lefta\na})=f_y^{s_\al\na\lefta\na}$. From definition, we see that $ys_\al<y$ implies that $wys_\al \prec_{w\fC_+}wy$.  Moreover, since $\be^\vee$ is a simple coroot with respect to $\fC_+$,  $\alpha^\vee=w\beta^\vee$ is a simple coroot with respect to $w\fC_+$, and $s_\al\na, \na$ are separated by the simple wall $H_{\al^\vee, 0}$. Therefore, Theorem \ref{thm:simplewall.gen} implies that $f_y^{s_\al\na\lefta \na}=q^{1/2}-q^{-1/2}$.
	
	Now if $\fC$ is general, we need to compute $f^{s_\al\na\lefta\na}_y\in K_{\bbC^*}(\pt)$ in 
	\[
	\stab^{\fC, T\calB, s_\al\na}_y=\stab^{\fC, T\calB, \na}_y+f^{s_\al\na\lefta \na}_y\stab^{\fC, T\calB, \na}_{ys_\al}.
	\]
	Write $\fC_+=v\fC$, then $ys_\al\prec_\fC y$ implies $vys_\al<vy$. Applying $v$ to the above identity, we get 
	\[
	\stab^{\fC_+, T\calB, s_\al\na}_{vy}=\stab^{\fC_+, T\calB, \na}_{vy}+f^{s_\al\na\lefta\na}_y\stab^{\fC_+, T\calB, \na}_{vys_\al}.
	\]
	By using the first part of the proof, we see that $f_{y}^{s_\al\na\lefta\na}=q^{1/2}-q^{-1/2}$. 
\end{proof}

By translation, we can compute the wall-crossing coefficients for any wall $H_{\alpha^\vee, n}$.
\begin{cor}\label{cor:general}
	Suppose $\na_1, \na_2$ are adjacent and  separated by the hyperplane $H_{\al^\vee, n}$. Moreover, assume that $(\la_1,\al^\vee)<n<(\la_2, \al^\vee), \forall \la_i\in \na_i$. Then for any chamber $\fC$,  the polarization $T\calB$  or $T^*\calB$, the coefficient $f^{\na_1\lefta \na_2}_y$ in Lemma \ref{lem:coefficients} is equal to $e^{-ny\al}(q^{1/2}-q^{-1/2})$. 
\end{cor}
\begin{proof} 
	There exists an integral weight $\varpi_\alpha\in \Lambda$ such that, $(\varpi_\alpha, \alpha^\vee)=1$. Indeed, if $\alpha$ is simple, take $\varpi_\alpha$ to be the corresponding fundamental weight. If $\al$ is not simple, there exists  $v\in W$ such that $v\alpha$ is simple. We can take $\varpi_\alpha=v^{-1}\varpi_{v\alpha}$, where $\varpi_{v\alpha}$ is the fundamental weight corresponding to the simple root $v\alpha$.  Then $\na_1':=\na_1-n\varpi_\alpha$ and $\na_2'=\na_2-n\varpi_\alpha$ are adjacent and  separated by $H_{\al^\vee, 0}$. Lemma \ref{lem:shift} and Theorem \ref{thm:nonsimplewall} imply that 
	\[f_y^{\na_1\leftarrow \na_2}=e^{-ny\alpha}f_y^{\na_1'\leftarrow \na_2'}=e^{-ny\al}(q^{1/2}-q^{-1/2}).\]
\end{proof}
\subsection{The affine Hecke algebra action}

We can further express  the wall-crossing formulas in terms of the affine Hecke algebra action. Denote $q_y=q^{\ell(y)}, \epsilon_y=(-1)^{\ell(y)}$.  
\begin{thm}\label{thm:wallcrossingHecke}
	For any $x, y\in W$, we have 
	\begin{eqnarray}\label{eq:Neg}
	\stm{x\na_+}{y}&=&q_x^{-1/2}T_x(\stm{\na_+}{yx});\\ 
	\stp{x\na_-}{y}&=&q_x^{1/2}(T'_{x^{-1}})^{-1}(\stp{\na_-}{yx}). \label{eq:Pos}
	\end{eqnarray}
\end{thm}
\begin{rem}
	This, together with Equations \eqref{eq:extremecases}, \eqref{eq:shift}, and Theorem~\ref{lem:actiononstable} completely determines the stable basis for the positive and negative chambers. 
	The stable basis associated to other chambers are determined by \S~\ref{subsec:WeylAction}, on top of those for the positive or negative chambers.
\end{rem}

\begin{proof}
	We first show that \eqref{eq:Neg} implies \eqref{eq:Pos}. 
	By the duality of stable bases \eqref{eq:dual} and the adjointness between $T_\al$ and $T'_\al$ \eqref{eq:adj}, we get
	\begin{align*}
	q_x^{1/2}(T_{x^{-1}}')^{-1}(\stp{\na_-}{yx})&=\sum_{v\in W}\langle 
	q_x^{1/2}(T_{x^{-1}}')^{-1}(\stp{\na_-}{yx}), \stm{x\na_+}{v}\rangle \stp{x\na_-}{v}\\
	&=\sum_{v\in W} \langle \stp{\na_-}{yx}, q_x^{1/2}T_{x}^{-1}(\stm{x\na_+}{v})\rangle\stp{x\na_-}{v}\\
	&\overset{\eqref{eq:Neg}}=\sum_{v\in W}\langle \stp{\na_-}{yx}, \stm{\na_+}{vx}\rangle\stp{x\na_-}{v}\\
	&=\stp{x\na_-}{y}.
	\end{align*}
	
	Now we prove \eqref{eq:Neg} by induction on  $\ell(x)$. If $x$ is a simple reflection, it follows from \eqref{eq:heckewallm} and \eqref{eq:Taction}. Now assume it holds for $x$. Let $\alpha$ be a simple root such that $xs_\alpha>x$. Then $\beta:=x\alpha>0$ and $s_\be=x s_\al x^{-1}$. Notice that the two alcoves $x\na_+$ and $xs_\alpha\na_+=s_{\be}x\na_+$ are adjacent alcoves separated by a wall on $H_{\beta^\vee,0}$, and for any $\la\in \na_+$, 
	\[(xs_\alpha\lambda,\beta^\vee)=(\lambda,s_\alpha x^{-1}\beta^\vee)=-(\lambda,\alpha^\vee)<0<(x\lambda,\beta^\vee).\]
	So we can use Theorem \ref{thm:nonsimplewall} with respect to the wall crossing $xs_\al\na_+\lefta x\na_+$, which corresponds to $s_\be x\na_+\lefta x\na_+. $
	
	If  $y\beta<0$, then $ys_\be<y$, and $ys_\be\succ_{\fC_-} y$. Since $yx \al=y\be<0$, $yxs_\al\succ_{\fC_-}yx$. Then 
	\begin{align*}
	\stm{xs_\al\na_+}{y}&\overset{\text{Thm.} \ref{thm:nonsimplewall}}=\stm{x\na_+}{y}\\
	&\overset{\text{induction}}=q_x^{-1/2}T_x(\stm{\na_+}{yx})\\
	&\overset{\text{Thm.}\ref{lem:actiononstable}}=q_x^{-1/2}T_x\left(q^{-1/2}T_\al(\stm{\na_+}{yxs_\al})\right)\\
	&=q_{xs_\al}^{-1/2}T_{x s_\al}(\stm{\na_+}{yxs_\al}).
	\end{align*}
	
	If $y\be>0$, then $ys_\be\prec_{\fC_-}y$, and $yxs_\al\prec_{\fC_-}yx$. We have
	\begin{align*}
	q_{xs_\al}^{-1/2}T_{xs_\al}(\stm{\na_+}{yxs_\al})&=q_{xs_\al}^{-1/2}T_xT_\al(\stm{\na_+}{yxs_\al})\\
	&\overset{\text{Thm.}\ref{lem:actiononstable}}=q_{xs_\al}^{-1/2}T_xT_\al[q^{-1/2}T_\al(\stm{\na_+}{yx})]\\
	&=q_{xs_\al}^{-1/2}T_x\left (q^{1/2}+(q^{1/2}-q^{-1/2})T_\al \right)(\stm{\na_+}{yx})\\
	&\overset{\text{Thm.}\ref{lem:actiononstable}}=q_{xs_\al}^{-1/2}T_x\left( q^{1/2}\stm{\na_+}{yx}+(q^{1/2}-q^{-1/2})q^{1/2}\stm{\na_+}{yxs_\al}\right)\\
	&=q_x^{-1/2}\left( T_x(\stm{\na_+}{yx})+(q^{1/2}-q^{-1/2})T_x(\stm{\na_+}{yxs_\al})\right)\\
	&\overset{\text{induction}}=\stm{x\na_+}{y}+(q^{{1/2}}-q^{-{1/2}})\stm{x\na_+}{yxs_\alpha x^{-1}}\\
	&\overset{\text{Thm.}\ref{thm:nonsimplewall}}=\stm{xs_\alpha\na_+}{y}.
	\end{align*}
\end{proof}

\begin{cor}\label{cor:TheOne}
If $\al$ is a simple root and $s_\al x>x$, then we have 
	\[
	q^{-1/2}T_\al(\stm{x\na_+}{y})=\stm{s_\al x\na_+} {ys_\al}, \quad q^{-1/2}T'_\al(\stp{s_\al x\na_-}{ys_\al})=\stp{ x\na_-}{y}. 
	\]
\end{cor}
\begin{proof}
	We have $(T_{s_\al x})^{-1}=(T_x)^{-1}T_\al^{-1}$. By Theorem \ref{thm:wallcrossingHecke}, we have 
	\begin{align*}
	q_x^{1/2}(T_x)^{-1}\stm{x\na_+}{yx^{-1}}=\stm{\na_+}{y}=q_{s_\al x}^{1/2}(T_{s_\al x})^{-1}(\stm{s_\al x\na_+}{yx^{-1}s_\al})=q_{ x}^{1/2}q^{1/2}(T_{ x})^{-1}T_\al^{-1}(\stm{s_\al x\na_+}{yx^{-1}s_\al}).
	\end{align*}
	Cancelling $q_x^{1/2}(T_x)^{-1}$, we get the first identity. The second one can be proved similarly. 
\end{proof}

\section{Affine braid group action and categorification of stable basis}
In this section we give a quick review of the affine braid group action on $D^{\ob}_{T}(T^*\calB)$, following \cite{BR,BMR2}, and then use it to define an integral lift of the stable basis $\{\stab^{+,T\calB,\nabla}_{w}\mid w\in W\}$. 

\subsection{The affine braid group}\label{sec:affWeyl}
We first recall the level-$p$ affine Weyl group action. 
In $\Lambda$, we consider the level-$p$ configuration of $\rho$-shifted affine hyperplanes. That is, 
for  a coroot $\alpha^\vee$ and $n\in\ZZ$, let the hyperplanes $H^p_{\alpha^\vee,n}$ be $\{\lambda\in\Lambda\mid\angl{\alpha^\vee,\lambda+\rho}= np\}$. The open facets are called $p$-alcoves, and the codimension one facets are called faces. 

There is a special alcove,  denoted by $A_0$, which is the alcove containing $(\epsilon-1)\rho$
for small $\epsilon>0$. It consists of those weights $\lambda$ such that $0< \angl{\lambda+\rho,\alpha^\vee}<p$ for all $\alpha\in\Phi^+$. 
This alcove is referred  as the fundamental alcove in \cite{BMR2}, which, to avoid confusion, will not be called the fundamental alcove in the present paper.

Let $Q$ be the root lattice.
We consider the level-$p$ affine Weyl group. It is the usual affine Weyl group $W_{\aff}:= W\ltimes Q$ as an abstract group, with a different action on the lattice $\Lambda$,  the level-$p$ dot-action, recalled as follows. Elements in $W$ acts via the usual dot-action $w:\la\mapsto w\bullet \la:=w(\la+\rho)-\rho$. Elements $\nu$ in the lattice act by $\lambda\mapsto\lambda+p\nu$. The group $W_{\aff}$ is generated
by reflections in affine hyperplanes $H_{\al^\vee,n}^p$. The set of faces in the closure  of $A_0$ will be denoted by $I_{\aff}$.
The $(W_{\aff},\bullet)$-orbits in the set of faces are canonically identified with $I_{\aff}$. 
Via this identification, the (Coxeter) generators of the group $W_{\aff}$ can be chosen to be the reflections in the faces of the alcove $A_0$.

We will consider the {\it extended affine Weyl group}\footnote{Here we follow the notations of \cite{BMR2}. In \cite{BR}, the notation $W_{\aff}$ is used for the extended affine Weyl group, and $W_{\aff}^{Cox}$ represents the non-extended one. The same convention holds for the corresponding affine Braid groups.} $W_{\aff}':=W\ltimes \Lambda$. Let $\Omega$ denote the stabilizer subgroup of $W'_{\aff}$ that fixes $A_0$. The group $W_{\aff}'$ is the semi-direct product $\Omega\ltimes W_{\aff}$ \cite[\S~3.5(a)]{Ca}. For example, for type $A_2$ root system (see Figure 1 below) the subgroup $\Omega$ is a cyclic group of order 3, cyclically permuting the 3 walls of the fundamental alcove.  
A fundamental domain of $\Lambda$ can be chosen to be $(W,\bullet)$-orbit of $A_0$. This means, for any $\lambda\in \Lambda$, there exists $\mu\in \Lambda$ and a Weyl group element $w\in W,$ such that  $\la-\mu\in w\bullet A_0$. 

Note that in the definition above, as abstract groups, $W_{\aff}$ and $W'_{\aff}$ do not depend on the level $p$. When $p$ is equal to 1, this action is the $\rho$-shift of the action used in \S~\ref{sec:recoll}, that is,  $H^{p=1}_{\alpha^\vee,n}$ is the $\rho$-shift of $H_{\alpha^\vee,n}$. In \S~\ref{subsec:CategoricalAction} and \S~\ref{sec:intst} below, we will focus on the $p=1$ case. The case when $p$ is a prime number will be relevant in \S~\ref{sec:Localization}.

We refer the readers to \cite[\S~1.1]{Lu80} for the {\it local} action of $W_{\aff}$, where for each $\alpha\in I_{\aff}$ and on each alcove $s_\alpha$ acts via reflection  along the wall labelled by $\alpha$. We follow \cite[\S~2.1.2]{BMR2} to define the local action of  $W_{\aff}'$. 
This depends on the choice of a character $\lambda_0\in A_0$, so that $W_{\aff}' \bullet \lambda_0$ is a free $W_{\aff}' \bullet$-orbit.  
Given  $v\in W'_{\aff}$ and $\lambda'=w\bullet \lambda_0\in W_{\aff}' \bullet \lambda_0$ with $w\in W'_{\aff}$,  define 
  $\lambda'\star v:=wv \bullet \lambda_0$. Note that $\bullet$ is a left action of $W_{\aff}' $ on $\Lambda$, and $\star$ is a right action of $W_{\aff}'$ only on the subset $W_{\aff}' \bullet \lambda_0$. 
For $\lambda\in W_{\aff}' \bullet \lambda_0$ and $w\in W'_{\aff}$, we say $w$ increases $\lambda$ if the following holds. Write $w=s_1\cdots s_n \omega$, with $s_i\in I_{\aff}$ and $\omega\in \Omega$. Then, $\lambda\star (s_1\cdots s_{i-1})<\lambda\star (s_1\cdots s_{i})$ in the standard order of $\Lambda$ for any $i=1,\dots, n$ \cite[\S~2.1.3]{BMR2}.

\begin{rem}\label{rmk:star_dot}
	\begin{enumerate}
		\item[(a)] 	We note the difference between $A_0$ and $\nabla_+$. We consider the map $\Lambda\to \Lambda\otimes_{\bbZ}\bbQ$, sending $\lambda\in\Lambda$ to $\frac{\lambda+\rho}{p}$. Note that $\lambda$ being in $A_0$ is equivalent to $\frac{\lambda+\rho}{p}$ being in the fundamental alcove $\nabla_+$ in the present sense. In turn, this is equivalent to $-\frac{\lambda+\rho}{p}\in\nabla_-$.
		Note that the map $\Lambda\to \Lambda\otimes_{\bbZ}\bbQ$, sending $\lambda\in\Lambda$ to $-\frac{\lambda+\rho}{p}$ intertwines the natural action of $\Lambda$ on $\Lambda_\bbQ$ and the level-$p$ action of $\Lambda$ on itself, composed with the abelian group automorphism $\Lambda\to\Lambda$, $\lambda\mapsto -\lambda$.
		\item[(b)] Moreover, we have  $x\frac{\lambda+\rho}{p}=\frac{x\bullet\lambda+\rho}{p}$ for any $\lambda$ and $x\in W$. 
		\item[(c)] For $\lambda_0\in A_0$ and $x\in W_{\aff}$, with $\lambda=\lambda_0\star x$, we have equivalently  $\lambda=x\bullet\lambda_0$ and $x^{-1}\bullet\lambda=\lambda_0$.
	\end{enumerate}
\end{rem}

To be consistent with \cite{SZZ} and \S~\ref{sec:recoll}, the roots of the Borel subalgebra $\fb$ are positive roots $\Delta^+$. Note that this is the opposite of the convention used in \cite[\S~1.1.1]{BMR2}, and hence results in an opposite lifting of the affine Weyl group into the affine braid group, as will be discussed in detail in Remark~\ref{rem:Left_right_J}(a).

Now we recall the affine braid group. 
The group $B_{\aff}$ is generated by $\wt{w}$ for $w\in W_{\aff}$ subject to the relation $\widetilde{wu}=\wt{w}\wt{u}$ when $l(wu)=l(w)+l(u)$ \cite[\S~3]{Mac}.  
There is a set theoretical lifting $C :W_{\aff}\to B_{\aff}$, sending the simple reflection $s_\alpha$ for $\alpha\in I_{\aff}$ to $\widetilde{s}_\alpha$, and any reduced decomposition $w=s_{\alpha_1}\cdots s_{\alpha_{k}}$ to $\wt{w}=\widetilde{s}_{\alpha_1}\cdots \widetilde{s}_{\alpha_{k}}$.  

On the extended affine Weyl group $W_{\aff}'$, the length function  is given by $l(w\omega)=l(w)$ for $\omega\in \Omega$. The extended affine Braid group $B_{\aff}'$ is generated by  $\wt{w}$ for $w\in W_{\aff}'$, subject to the relation $\widetilde{wu}=\wt{w}\wt{u}$ when $l(wu)=l(w)+l(u)$. As $\Omega$ permutes $I_{\aff}$, we have  $B_{\aff}'=B_{\aff}\rtimes\Omega$. A smaller set of generators of $B_{\aff}'$ can be chosen to be $\{\wt{s}_\alpha\mid \alpha\in I_{\aff}\}$ and $\Omega$.

The set of finite simple roots $\Sigma$ can be naturally identified with a subset of $I_{\aff}$ via their corresponding simple reflections. Alternatively, the group $B_{\aff}'$ is generated by simple finite reflections $\wt{s}_\alpha$ for $\alpha\in\Sigma$ and  translations  $\Lambda^+\subset \Lambda$ consisting of $\lambda\in \Lambda$ with $\langle\alpha^\vee,\lambda\rangle>0$ for all positive coroots $\alpha^\vee$.

For any  $w\in W$,
let $w=s_1\dots s_k$ be a reduced  decomposition of $w$. 
Then, in $B_{\aff}'$ we have the element $\widetilde{s}_1\dots \widetilde{s}_k$ via $\Sigma\subseteq I_{\aff}$. This element  is independent of the reduced word decomposition, hence is a well-defined lifting $\wt{w}\in B'_{\aff}$. Sending $w\in W$ to $\wt{w}$ defines a set-theoretical lifting $W\to B_{\aff}\subseteq B'_{\aff}$. Since $w^{-1}=s_k\dots s_1$,  we have $\wt{w^{-1}}=\widetilde{s}_k\dots \widetilde{s}_1$, which is different from $(\wt{w})^{-1}=(\widetilde{s}_k)^{-1}\dots (\widetilde{s}_1)^{-1}$.
On the other hand, the lifting of $\Lambda\subseteq W'_{\aff}$ is a subgroup of $B_{\aff}'$ \cite[\S~3.3]{Mac}. Therefore for any $\lambda\in A$, we will not distinguish between the group inverse taken in $\Lambda$  and the group inverse taken in $B_{\aff}'$. For this reason, we will denote the lifting of $\lambda$ still by $\lambda$. Similar property holds for the subgroup $\Omega\subseteq W'_{\aff}$.

\subsection{A categorical affine braid group action}\label{subsec:CategoricalAction}
Let $G_\bbZ$ be a split $\bbZ$-form of the complex algebraic group $G$. Let $A_\bbZ\subseteq B_\bbZ\subseteq G_\bbZ$ be a maximal torus and a Borel subgroup, which we assume to have base changes $A\subset B\subset G$. Let $\fh_\bbZ\subset\fb_\bbZ\subset\fg_\bbZ$ be   the corresponding Lie algebras. 
Similar to the setup of \S~\ref{sec:recoll}, we have the integral forms $\calB_\bbZ$ and $T^*\calB_\bbZ$, etc.

In \cite{BR}, an action of the extended affine braid group on the category $D^{\ob}_{T_\bbZ}(T^*\calB_\bbZ)$ is constructed, where  elements of $B'_{\aff}$ act by Fourier-Mukai transforms. Moreover, for any algebraically closed field $k$, this affine braid group action can be base changed to $k$. 
More precisely, we have the following. The base change of the groups and varieties from $\bbZ$ to $k$ are denoted by $A_k\subseteq B_k\subseteq G_k$, $T^*\calB_k$, etc. When $k=\bbC$, we will omit the subscripts. This is consistent with the notations introduced in \S~\ref{sec:recoll}.
\begin{thm}\cite[Theorem~1.3.1, Proposition~1.4.3, Theorem~1.6.1]{BR}\label{thm:BR}
There exists a right action  of $B_{\aff}'$ on the category $D^{\ob}_{T_\bbZ}(T^*\calB_\bbZ)$, denoted by $J_b^\bbZ$ with $b \in B_{\aff}'$, such that the generator $J_s^\bbZ$ for $s\in \Sigma$ acts via a Fourier-Mukai transform with kernel the structure sheaf of a certain subvariety of $T^*\calB_\bbZ\times T^*\calB_\bbZ$ (given explicitly in \cite[\S~1.3]{BR}), and the generator $J_\lambda^\bbZ$ for $\lambda\in \Lambda$ acts via tensoring by the line bundle $\calL_\lambda$.

Similarly, for any algebraically closed field $k$, applying 
$\_\otimes^L_\bbZ k$ to the action above, we have a right action  of $B_{\aff}'$ on the category $D^{\ob}_{T_k}(T^*\calB_k)$, denoted by $J^k_b$ with $b\in B_{\aff}'$, such that the generator $J_s^k$ for $s\in \Sigma$ acts via a Fourier-Mukai transform with kernel the structure sheaf of certain subvariety of $T^*\calB_k\times T^*\calB_k$, and the generator $J_\lambda^k$ for $\lambda\in \Lambda$ acts via tensoring by the line bundle $\calL_\lambda$.
\end{thm}

\begin{rem}\label{rem:Left_right_J}
\begin{enumerate}
\item[(a)] 
Using the local affine Weyl group action, the order on $\Lambda$, which is determined by a choice of a dominant dual Weyl  chamber $\Lambda^+$, identifies the abstractly defined $B_{\aff}$ above with the topologically defined $\pi_1(\fh^*_{\bbC,reg}/(W_{\aff},\bullet))$. Then the lifting $C:W_{\aff}\to B_{\aff}$ here is characterized as the set of positive paths going from any alcove to an alcove above it in the natural order \cite[Lemma~1.8.1(c)]{BM}.

As mentioned above, the choice of positive roots in the present paper is  opposite to that in \cite{BMR2, BR}, and hence the anti-dominant chamber $-\Lambda^+$ in our setting consists of ample line bundles. 
Consequently, the order on $\Lambda$ considered here is opposite to that of \cite{BMR2, BR}. 
For comparison, denote the order of \cite{BMR2, BR} by $>^{BMR}$. 
Then, for any $s\in I_{\aff}$, $\lambda\star s>\lambda$ is equivalent to $\lambda\star s<^{BMR}\lambda\star s\star s=\lambda$. 
Taking the corresponding lifting of $s$ to $B_{\aff}'$, which is denoted by $\tilde{s}$  in the present notation, we have that the lifting taken in \cite{BMR2, BR} is $(\wt s)^{-1}$.  
\item[(b)]
The action of $B_{\aff}$ in \cite{BR} is a left action. As we will not be considering any left actions in this section, we use \cite[Remark~2.5.2]{BR} and the paragraph after \cite[Corollary~2.5.3]{BR} to translate the left action therein to a right action. The effect coincides with taking the opposite lifting as in part (a) above. Therefore, 
when $k=\bbC$, taking the Grothendieck group, the above actions of the generators $J_{\wt{s}_\alpha}$ and $J_\lambda$ for $\alpha\in \Sigma$ and $\lambda\in\Lambda$ of $B'_{\aff}$ on $K_T(T^*\calB)$ from Theorem~\ref{thm:BR} are equal to the operators $q^{-1/2}T^{\oR}_\alpha$ and $\calL_\lambda$ in \S~\ref{sec:heckeaction}, respectively. In particular, the action factors through the action of the affine Hecke algebra $\bbH$ considered in \S~\ref{sec:recoll}.

\item[(c)] To remind ourselves that this is a right action, and also to be consistent with the notations in \S~\ref{sec:recoll}, we will also write the operator $J_b$ as $J_b^{\oR}$. By our convention of a right action, for any $w\in W$, via the lifting $W\to B_{\aff}$, we have  $J^{\oR}_{\wt{w}}=J^{\oR}_{\widetilde{s}_k}\circ \cdots\circ J^{\oR}_{\widetilde{s}_1}$  if $w=s_1\dots s_k$ is a reduced  decomposition of $w$.  
\end{enumerate}
\end{rem}

\subsection{Integral form of stable bases}\label{sec:intst}
Now we define the objects $\fstab^\bbZ_{\lambda}(y)\in D^{\ob}_{T_\bbZ}(T^*\calB_\bbZ), ~y\in W,~ \la\in \Lambda_\bbQ$ in the interior of an alcove, inductively as follows:
\begin{flalign}\label{eqn:Z1}
\fstab^{\bbZ}_{\lambda_0}(id)=\calL_{-\rho}\otimes\calO_{T^*_{id}\calB_\bbZ}, &\quad \la_0\in \na_-,  \\
\label{eqn:Z2}
\fstab^{\bbZ}_{\lambda_0}(w)=J_{\wt{w}}^{\oR}\fstab^{\bbZ}_{\lambda_0}(id), & \quad \la_0\in \na_-, \\ \label{eqn:Z3}
\fstab^\bbZ_{\lambda}(y)=(J_{\wt{x}}^{\oR})^{-1}\fstab^\bbZ_{\lambda_0}(yx), &\quad y,x\in W,~ x\lambda_0=\lambda, ~\la_0\in \na_-,\\ \label{eqn:Z4}
\fstab^\bbZ_{\lambda}(y)=e^{-y\mu}J^{\oR}_{\mu}\fstab^\bbZ_{\lambda_1}(y), &\quad y\in W, ~\mu+\lambda_1=\lambda,~ \mu\in Q,~ \lambda_1\in W\nabla_-. 
\end{flalign}
According to \ref{thm:BR}, the functor $J^{\oR}_{\mu}$ in \eqref{eqn:Z4} is twisting by $\calL_\mu$. From the definitions, $\fstab^\bbZ_{\lambda}(y)$ only depends on the alcove $\nabla$ containing $\lambda$, so sometimes it is also denoted by $\fstab^\bbZ_{\na}(y)$.

Now we are ready to state the main result of this section. 
\begin{thm}\label{thm:IntStable}
Applying $-\otimes^L_{\bbZ}\bbC$ to $\fstab^\bbZ_{\na}(w)\in D^{\ob}_{T_\bbZ}(T^*\calB_\bbZ)$, and taking the class in the Grothendieck group, we get  $\calL_{-\rho}\otimes\stab^{+,T\calB,\na}_w\in K_T(T^*\calB)$. 
\end{thm}
\begin{proof}
We need to verify  \eqref{eqn:Z1}\eqref{eqn:Z2}\eqref{eqn:Z3}\eqref{eqn:Z4} inductively. 
The first equation is a direct consequence of Equation \eqref{eq:extremecases}.
Keeping in mind Theorem~\ref{thm:BR} and the relations $[J_{\wt{s}_\alpha}^\bbC]=q^{-1/2}T^{\oR}_\alpha$ and $T^{\oR}_\alpha=\calL_{-\rho}T'_\alpha\calL_{\rho}$, then the second equation follows directly  from Theorem~\ref{lem:actiononstable}. The third one follows from  
Theorem~\ref{thm:wallcrossingHecke}, and the last one follows from  \eqref{eq:shift}.
\end{proof}

\section{Verma modules in prime characteristic}\label{sec:Localization}
In this section, we prove that, when base changed to a field  $k$ of positive characteristic, the integral form of stable basis  defined in \S~\ref{sec:intst} agrees with the Verma modules of $U(\fg_k)$ under the localization equivalence of \cite{BM,BMR1}.

\subsection{Reminder on localization in prime characteristic}\label{subsec:babyV}
We start by a brief review of \cite{BMR1}. Now we assume $k$ is an algebraically closed field of characteristic $p$ greater than the Coxeter number of $G$. 
For any $k$-variety $X$, its Frobenius twist is denoted by $X^{(1)}$.
If   $X$ is defined over $\bbF_p$, then $X^{(1)}\cong X$ as abstract $k$-varieties. However, the natural map $X\to X^{(1)}$ is not an isomorphism.
Below we will take $X$ to be $\calB_k$, $T^*\calB_k$, etc. We will often omit the subscript $k$, when  there is a Frobenius twist. For example, ${\calB_k}^{(1)}$ will be simply denoted by $\calB^{(1)}$.   

Differentiating a character of $A_k$ defines a map $d:\Lambda\to \fh_k^*$. We call a character $\lambda\in\fh^*_k$ integral if it lies in the image of $d:\Lambda\to \fh^*_k$. An integral character is said to be regular if it does not lie on any walls $H^p_{\alpha^\vee,n}$. 
For $\lambda\in\Lambda$, on the flag variety $\calB_k$, consider the ring of $\calL_\lambda$-twisted differential operators $\calD^\lambda$, which is a sheaf of algebras. Its central subalgebra, i.e., the sheaf of functions on $T^*\calB^{(1)}$, makes any coherent $\calD^\lambda$-module  naturally a coherent sheaf on $T^*\calB^{(1)}$, and moreover, $\calD^\lambda$ is an Azumaya algebra on $T^*\calB^{(1)}_k$ \cite[\S~2.3]{BMR1}. 

There exists a vector bundle  $\calE^s$ with the following properties.
\begin{enumerate} 
\item $\calE nd_{T^*\calB^{(1)\wedge}}(\calE^s)\cong \calD^{\lambda_0}|_{T^*\calB^{(1)\wedge}}$, where $T^*\calB^{(1)\wedge}$ is the formal neighbourhood of the zero-section $\calB^{(1)}$. That is, the Azumaya algebra $\calD^{\lambda_0}$ splits on this formal neighbourhood \cite[Theorem~5.1.1]{BMR1}. 
\item $\calE^s$ is $T_k$-equivariant \cite[\S~5.2.4]{BM}, where $T_k=A_k\times k^*. $
\item Define $\calA^{\lambda_0}$ to be $End_{T^*\calB^{(1)}}(\calE^s)$, a $\calO({\calN^{(1)}})$-algebra endowed with a compatible $A_k\times(\Gm)_k$-action. Then there exists   a localization functor $\gamma^{\lambda_0}: D^{\ob}\Mod^{\gr}(\calA^{\lambda_0},A_k)\cong D^{\ob}_{T_k}\Coh(T^*\calB^{(1)}_k)$ \cite[Theorem~5.1.1]{BM}. \footnote{The bundle $\calE^s$ is not equal but equidecomposable to the bundle $\calE$ in \cite{BM}. See \cite[Corollary 1.6.8]{BM}. Hence, the algebra $\calA^{\lambda_0}$ considered here is not equal but Morita equivalent to the algebra $A$ in \cite{BM}.} 
\item The completion  $(\calA^{\lambda_0})^\wedge_0$ is related to the Lie algebra $\fg_k$ (see \eqref{eqn:comp} below).
\end{enumerate}

The algebra $U(\fg_k)$ has two central subalgebras, the Frobenius center, which  is naturally identified with $\calO(\fg^{*(1)})$, and the Harish-Chandra center, which is identified with $\calO(\fh^*_k/(W,\bullet))$ under the choice of the Borel subgroup $B_k$ and the maximal torus $A_k$. A point $W\bullet d\lambda\in\fh_k^*/(W,\bullet)$ defines a maximal ideal in $\calO(\fh^*_k/(W,\bullet))$, which is a central subalgebra of $U(\fg_k)$. The quotient of $U(\fg_k)$ by the central ideal is
denote by $U(\fg_k)^\lambda$, that is $U(\fg_k)^\lambda:=U(\fg_k)\otimes_{\calO(\fh^*_k/(W,\bullet))}k_\lambda$. Similarly $\chi\in\calN^{(1)}$ defines a central ideal, and the completion of $U(\fg_k)^{\lambda_0}$ at this ideal  is denoted by
 $U(\fg_k)^{\lambda_0}_\chi$.
Then, \begin{equation}\label{eqn:comp}(\calA^{\lambda_0})^\wedge_0\cong U(\fg_k)_0^{\lambda_0}.\end{equation}

The vector bundle $\calE^s$ above is not unique. Different choices are related by $\calD^\lambda$ \cite[Remark~5.2.2]{BMR1}. Nevertheless, as has been done in \cite{BR,BM}, the  property \eqref{eqn:verma} below about baby Verma modules fixes the choice \cite[Remark 1.3.5]{BMR2}.

For the Borel subalgebra $\fb_k$ and for any $\chi\in\calN^{(1)}$ and $\lambda\in\Lambda$, consider the $U(\fb_k)$-character $k_\lambda$ obtained via the projection $\fb\to\fh$ composing with the $\fh$-character $k_\lambda$. Recall that  the Verma module of $U(\fg_k)$ associated to the Borel $\fb$ and $\lambda\in\Lambda$ is $Z^{\fb}(\lambda):=U(\fg_k)\otimes_{U(\fb_k)}k_\lambda$, and the baby Verma module is $Z^{\fb}_\chi(\lambda):=U_\chi(\fg_k)\otimes_{U(\fb_k)}k_{\lambda}$, where $U_\chi(\fg_k)$ is the quotient of $U(\fg_k)$ by the central ideal $\chi\in \calN^{(1)}\subseteq \fg^{*(1)}$.

Recall that under the identification  $\calB_k\cong G_k/B_k$, the $T_k$-fixed points on $T^*\calB^{(1)}$ are identified with the Weyl group elements,  with $B_k$ corresponding to $id$, whose skyscraper sheaf is denoted by $k_{id}$. 
Then \cite[Example~5.3.3.(0)]{BMR1}
\begin{equation}\label{eqn:verma}\gamma^{\lambda_0} Z_0^{\fb}(\lambda_0+2\rho)\cong k_{id}. \end{equation} 

The isomorphism \eqref{eqn:comp} defines a $T_k$-action on $U(\fg_k)^{\lambda_0}_0$ making it an equivariant isomorphism. Taking completion of $Z^{\fb}(\lambda_0+2\rho)$ with respect to the maximal ideal $W\bullet d\lambda\in\fh_k^*/(W,\bullet)$ of the central subalgebra $\calO(\fh_k^*/(W,\bullet))$ defines a module over $U(\fg_k)^{\lambda_0}_0$. 
The $T_k$-action on  $U(\fg_k)^{\lambda_0}_0$ provides $T_k$-actions on  $Z_0^{\fb}(\lambda_0+2\rho)$ and $Z^{\fb}(\lambda_0+2\rho)$, making \eqref{eqn:verma} equivariant.  Taking $\Gm$-finite vectors in $Z^{\fb}(\lambda_0+2\rho)$, we get an $A_k$-Koszul graded module of $\calA^{\lambda_0}$, which, without causing confusions, will still be denoted by $Z^{\fb}(\lambda_0+2\rho)$. 
 Then, under the equivalence in property (3) of \S~\ref{subsec:babyV} above,  we have the isomorphism 
\begin{equation}
\gamma^{\lambda_0} Z^\fb(\lambda_0+2\rho)\cong 
\calO_{T^*_{id}\calB^{(1)}}.
\end{equation}

\subsection{Localization and the affine braid group action}\label{subsec:LocAffBraid}
We  collect
 a few preliminary results, which are direct consequences of \cite{BMR2} reviewed above.
Let us  consider the category $\Mod_0 U(\fg_k)^\lambda$ of finitely generated modules of $U(\fg_k)^\lambda$ on which the Frobenius center $\calO(\fg^{*(1)})$ acts by the generalized character $0\in \fg^{*(1)}$.
For $\lambda$, $\mu\in\Lambda$, we define $T_\lambda^\mu:\Mod_0 U(\fg_k)^\lambda\to\Mod_0 U(\fg_k)^\mu$ sending $M$ to $[V_{\mu-\lambda}\otimes M]_\mu$. Here $V_{\mu-\lambda}$ is a finite dimensional representation with extremal weight $\mu-\lambda$, and $[-]_\mu$ means taking the component supported on the point $\mu$ in $\fh^*/\!/W$. 
Assume $\nu$ lies in a codimension-one wall of the facet of the alcove containing $\mu$. Define \[R_{\mu|\nu}:=T^\mu_\nu T_\mu^\nu: \Mod_0 U(\fg_k)^\mu\to \Mod_0 U(\fg_k)^\mu.\]
Indeed, $R_{\mu|\nu}$ depends only on the wall, not the character $\nu$ itself.
Taking mapping cone of the adjunction, we define \[\Theta_{\mu|\nu}:=cone(id\to R_{\mu|\nu}).\]
When $\mu$ is in the alcove $A_0$, then the wall on which $\nu$ lies is labeled by some $\alpha\in I_{\aff}$, therefore we will also denote $\Theta_{\mu|\nu}$ by $\Theta_\alpha$ in this case. 
For $\omega\in\Omega$, we write $T_{\lambda_0}^{\lambda_0\star \omega}:\Mod_0 U(\fg_k)^{\lambda_0}\to  \Mod_0 U(\fg_k)^{\lambda_0}$ simply as $\Theta_{\omega}$. Note that this is possible due to the fact that $\lambda_0$ and $\lambda_0\star \omega$ have the same central character in $\fh^*_k/(W,\bullet)$.
\begin{rem}\label{rmk:inverse_BMR}
For $\alpha\in I_{\aff}$, the functor $\Theta_\alpha$   is the same as that in \cite{BMR2}, and is equal to $\mathbf{I}_{\alpha}^{-1}$ in \cite{BR}. Similarly, for $\omega\in \Omega$, $\Theta_\omega$ in the present paper is $\bfI_{\omega^{-1}}$ in \cite{BR}. This inversion is also noticed in \cite[Remark~2.5.2]{BR}. In the present paper, we adopt the notation from \cite{BMR2} so as  to consider right actions throughout. This is consistent with Remark \ref{rem:Left_right_J}.
\end{rem}

The following is essentially \cite[Theorem~2.1.4]{BMR2} and \cite[Theorem~ 2.5.4]{BR}, taking into account the correction in \cite[Remark~2.5.2]{BR} and the translation between the left and right actions in Remark~\ref{rem:Left_right_J}.

\begin{thm}\label{thm:loc_wall} Let $\lambda_0$ be in the alcove $A_0$.
	\begin{enumerate} 
		\item[(i)] The functors $\Theta_\alpha$ for $\alpha\in I_{\aff}$ and $R_\omega$ for $\omega\in\Omega$ induce  right actions    of the group  $B_{\aff}'$ on the categories $D^{\ob}\Mod_{0}U(\fg_k)^{\lambda_0}$, $D^{\ob}\Mod_0^{\gr}(U(\fg_k)^{\lambda_0},A_k)$, $D^{\ob}\Mod^{\gr}(U(\fg_k)^{\lambda_0}_0,A_k)$, and $D^{\ob}\Mod^{\gr}(\calA^{\lambda_0},A_k)$. 
		\item[(ii)] 	For any	$b\in B_{\aff}'$, write the corresponding  auto-equivalence as $\Theta_b^{\oR}$. The localization functor $\gamma^{\lambda_0}$ intertwines these two $B_{\aff}'$-actions. That is, for any $b\in B_{\aff}'$, the following diagram commutes
		\[\xymatrix{
			D^{\ob}\Mod_0U(\fg_k)^{\lambda_0}\ar[rr]^{\gamma^{\lambda_0}}&&D^{\ob}\Coh_{\calB^{(1)}}T^*\calB^{(1)}\\
D^{\ob}\Mod_0U(\fg_k)^{\lambda_0}\ar[u]^{\Theta^{\oR}_b}\ar[rr]^{\gamma^{\lambda_0}}&& D^{\ob}\Coh_{\calB^{(1)}}T^*\calB^{(1)}\ar[u]^{J^{\oR}_b}.
		}\]
Similar conclusions hold for the completed or  graded versions: \[D^{\ob}\Mod_0^{\gr}(U(\fg_k)^{\lambda_0},A_k), \quad D^{\ob}\Mod^{\gr}(U(\fg_k)^{\lambda_0}_0,A_k), \quad \text{ and } \quad  D^{\ob}\Mod^{\gr}(\calA^{\lambda_0},A_k). \]
	\end{enumerate}
\end{thm}

Note that under this identification of the two affine braid group actions, $\Pic(T^*\calB)\cong\Lambda$ acts on $\Lambda$ as in the level-$p$ affine Weyl group action $\bullet$.

The un-graded version of the following has been well-known
(see, e.g., \cite[Lemma~4.7]{Janz}). Here we give a different proof, which also holds in the $A_k$-Koszul-bigraded setting.

\begin{lem}\label{lem:babyV&T}
Assuming $\mu$ is regular. 
Let $\alpha\in I_{\aff}$ be such that $\dot{\mu}:=\mu\star s_\alpha>\mu$, then we have as objects in $D^{\ob}\Mod^{\gr}(\calA^{\lambda_0},A_k)$
\begin{align*}\Theta_{\wt{s}_\alpha}^{-1}(Z^{\fb}_0(\mu+2\rho))\cong Z^{\fb}_0(\dot{\mu}+2\rho),\\
\Theta_{\wt{s}_\alpha}^{-1}(Z^{\fb}(\mu+2\rho))\cong Z^{\fb}(\dot{\mu}+2\rho).\end{align*}
\end{lem}
\begin{proof}
The usual equivariant D-module calculation as in \cite{BB} (see also \cite[\S~11.2.11]{BF}) yields that the global section of the delta-D-module at $id\in\calB_k$ gives the Verma module over $U(\fg_k)$. More precisely, for any $\lambda$ regular and integral, let $\delta_\fb^\lambda$ be the module of $\calD^\lambda$  given by the $\delta$-distributions at $id\in\calB_k$. Under the equivalence $R\Gamma_{\calD^\lambda,0}:D^{\ob}(\Coh_{0}\calD^\lambda)\overset{\cong}\longrightarrow D^{\ob}(\Mod_{0}U(\fg_k)^\lambda)$ \cite[Theorem 3.2]{BMR1}, 
\begin{equation}\label{eqn:SectionVerma}
R\Gamma_{\calD^\lambda,0}(\delta_\fb^\lambda)\cong Z^\fb(\lambda+2\rho).
\end{equation}
Similarly, let $\delta^\lambda_{\fb,0}$ be the central reduction of $\delta_\fb^\lambda$ with respect to the sheaf of ideals corresponding to $\calB^{(1)}\subseteq T^*\calB^{(1)}$ in the center of $\calD^\lambda$. Then we have \cite[\S~3.1.4]{BMR1}
\begin{equation}\label{eqn:SectionBabyVerma}
R\Gamma_{\calD^\lambda,0}(\delta_{\fb,0}^\lambda)\cong Z^\fb_0(\lambda+2\rho).
\end{equation}

By \cite[Theorem~2.2.8]{BMR2} (taking into account the correction in \cite{BR} in Remark~2.5.2 and the paragraph afterwards), the following diagram commutes 
\[
\xymatrix{
D^{\ob}\Coh_0\calD^\mu\ar[rr]^{-\otimes\calL_{\dot{\mu}-\mu}}\ar[d]^{R\Gamma_{\calD^\mu,0}}&&D^{\ob}\Coh_0\calD^{\dot{\mu}}\ar[d]^{R\Gamma_{\calD^{\dot{\mu}},0}}\\
D^{\ob}\Mod_0 U(\fg_k)^\mu\ar[rr]_{\Theta_{\alpha}^{-1}}&&D^{\ob}\Mod_0 U(\fg_k)^{\dot\mu}
.}
\]
Here we have used the assumption that $\dot{\mu}=\mu\star s_\alpha>\mu$. 
By definition, \[\delta_{\fb, 0}^{\dot{\mu}}\cong \delta_{\fb, 0}^{\mu}\otimes\calL_{\dot{\mu}-\mu}.\]
By \eqref{eqn:SectionBabyVerma}, $R\Gamma_{\calD^\mu,0}(\delta_{\fb, 0}^{{\mu}})=Z^\fb_0(\mu+2\rho)$, and $R\Gamma_{\calD^{\dot{\mu}},0}(\delta_{\fb, 0}^{\dot{\mu}})=Z^\fb_0(\dot{\mu}+2\rho)$. 
Therefore, we have $\Theta_{\alpha}^{-1}(Z^{\fb}_0(\mu+2\rho))\cong Z^{\fb}_0(\dot{\mu}+2\rho)$. Note that this isomorphism is equivariant with respect to the $A_k$ and Koszul gradings.

About the assertion on Verma modules, we follow the same argument using \eqref{eqn:SectionVerma}, with the commutative square above replaced by their counterparts for the completed algebras. We get an isomorphism between the completions of $\Theta_{\alpha}^{-1}(Z^{\fb}(\mu+2\rho))$ and $Z^{\fb}(\dot{\mu}+2\rho)$, which is equivariant with respect to the $A_k$ and Koszul gradings.  Then, taking the subspaces of finite vectors with respect to the  Koszul grading gives  the required isomorphism.
\end{proof}

The following is a direct consequence of the existence of the affine braid group action on the representation category.
\begin{cor}\label{cor:Theta_Z}
If $\lambda_0\in A_0$, then  for any $w\in W$ we have $\Theta_{\widetilde{w^{^{-1}}}}^{\oR}(Z^{\fb}(\lambda_0+2\rho))=Z^{\fb}(\lambda_0\star w^{-1}+2\rho)$.
Similar conclusion holds  for the baby Verma modules. 
\end{cor}
\begin{proof}
We have $\lambda_0=\lambda_0\star w^{-1}\star w$. Clearly $w$ increases $\lambda_0\star w^{-1}$ by assumption. 
Then, the statement follows directly by  applying  Lemma~\ref{lem:babyV&T} iteratively to $\Theta_{\widetilde{w^{^{-1}}}}^{\oR}=\Theta^{\oR}_{\widetilde{s}_1}\circ \cdots\circ \Theta^{\oR}_{\widetilde{s}_k}$ for any reduced  decomposition $w=s_1\dots s_k$ of $w$, similar to the calculation done in Remark~\ref{rem:Left_right_J}(c). We get $(\Theta_{\widetilde{w^{^{-1}}}}^{\oR})^{-1}Z^\fb(\lambda_0\star w^{-1})=(\Theta^{\oR}_{\widetilde{s}_1}\circ \cdots\circ \Theta^{\oR}_{\widetilde{s}_k})^{-1}Z^\fb(\lambda_0\star w^{-1})=(\Theta^{\oR}_{\widetilde{s}_k})^{-1}\circ \cdots\circ (\Theta^{\oR}_{\widetilde{s}_1})^{-1}Z^\fb(\lambda_0\star w^{-1})=Z^\fb(\lambda_0\star w^{-1}\star w)$.
\end{proof}

Motivated by  \cite[p.870]{BM}, for any $\lambda$ in the $ W_{\aff}'$-orbit of $\lambda_0$, we define the localization functor 
\[\gamma^\lambda:D^{\ob}\Mod_0U(\fg_k)^{\lambda_0}\to D^{\ob}\Coh_{\calB^{(1)}}T^*\calB^{(1)},\] which is determined by the property that $\gamma^{\lambda\star w}=\gamma^{\lambda}\circ \Theta_{\widetilde{w^{^{-1}}}}^{\oR} $ for $w\in W'_{\aff}$, as long as $w$ increases $  \lambda$.
Localization functor associated to singular $\la$ has also been studied in \cite{BMR2}, although we will not explicitly use this.
In particular, if $\lambda_0\in A_0$,  and $w\in W$ such that $\lambda=\lambda_0\star w$, then using the fact that $\lambda_0=\lambda_0\star w\star w^{-1}$ and that $w^{-1}$ increases $\lambda_0\star w$, we get 
			\begin{equation}\label{eqn:L_Theta}
			\gamma^{\lambda_0}= \gamma^{\lambda}\circ \Theta_{\widetilde{w}}^{\oR}=J_{\widetilde{w}}^{\oR}\circ\gamma^{\lambda}.
			\end{equation}

\subsection{A categorification of the stable basis via the Verma modules}

Now we are ready to state and prove the main theorem of this section.

\begin{thm}\label{Thm:cat}
Let $k$ be an algebraically closed field of characteristic $p$, and $p$ is greater than the Coxeter number. Assume $\lambda$ is regular and integral, then in $D^{\ob}_{T_k}(T^*\calB_k^{(1)})$, we have isomorphisms	\[e^{\rho}\fstab^k_{-\frac{\lambda+\rho}{p}}(y)\cong \gamma^{\lambda} Z^{\fb}(y\bullet \lambda+2\rho).\]
\end{thm}
The left hand side is the base changed integral form of the stable basis defined in \S \ref{sec:intst}, not to be confused with the complex one. 
The relation between these two is given by Theorem~\ref{thm:IntStable}.
Using the isomorphism of $k$-varieties $T^*\calB_k\cong T^*\calB^{(1)}$, we consider $T^*\calB^{(1)}$ as the base-change of $T^*\calB_\bbZ$ to the field $k$.

\begin{proof}
As the objects $\fstab^k_{-\frac{\lambda+\rho}{p}}(y)$ are defined inductively, we prove this theorem inductively. 

Firstly, we prove the case when $\lambda=\lambda_0\in A_0$. We have $\gamma^{\lambda_0}Z^\fb(\lambda_0+2\rho)=\calO_{T^*_{id}\calB}=\calL_\rho\otimes\fstab^k_{-\frac{\lambda_0+\rho}{p}}(id)=e^\rho\fstab^k_{-\frac{\lambda_0+\rho}{p}}(id)$. Here $id\in \calB_k$  corresponds to the Borel subalgebra $\fb$ that has been used to define the  Verma modules. For $w\in W$, Remark~\ref{rmk:star_dot}(c) yields that $Z^\fb(w\bullet \lambda_0+2\rho)=Z^\fb(\lambda_0\star w+2\rho)$, so 
\begin{align*}
&\gamma^{\lambda_0}Z^\fb(\lambda_0\star w+2\rho)\overset{Cor~\ref{cor:Theta_Z}}=\gamma^{\lambda_0}\Theta^{\oR}_{\wt w}Z^\fb(\lambda_0+2\rho)
\overset{Thm~\ref{thm:loc_wall}(ii)}=J^{\oR}_{\wt w}\gamma^{\lambda_0}Z^\fb(\lambda_0+2\rho)\\
=&J^{\oR}_{\wt w}e^\rho\fstab^k_{-\frac{\lambda_0+\rho}{p}}(id)
\overset{\eqref{eqn:Z2}}=e^{\rho}\fstab^k_{-\frac{\lambda_0+\rho}{p}}(w).
\end{align*}
Here the third equality follows from the initial case above. 

Secondly, we prove this for general $\lambda\in W\bullet A_0$, a fundamental domain of the action of the Picard group. We assume $\lambda=x\bullet\lambda_0$ with $\lambda_0\in A_0$ and $x\in W$. We have 
\begin{align*}
&\fstab^k_{-\frac{\lambda+\rho}{p}}(y)\overset{Rmk~\ref{rmk:star_dot}(b)}=\fstab^k_{-x\frac{\lambda_0+\rho}{p}}(y)
\overset{ \eqref{eqn:Z3}}=(J^{\oR}_{\wt{x}})^{-1}\fstab_{-\frac{\lambda_0+\rho}{p}}(yx)\\
=&(J^{\oR}_{\wt x})^{-1}e^{-\rho}\gamma^{\lambda_0} Z^{\fb}((yx)\bullet\lambda_0+2\rho)
\overset{\eqref{eqn:L_Theta}}=e^{-\rho}\gamma^{\lambda}Z^{\fb}((yx)\bullet\lambda_0+2\rho)
=e^{-\rho}\gamma^{\lambda}Z^{\fb}(y\bullet \lambda+2\rho).
\end{align*}
Here the third equality follows from the first step above;
The last equality follows from  $(yx)\bullet\lambda_0=y\bullet \lambda$.

Lastly, we consider the case of general $\lambda$. There is an $\mu\in \Lambda$, and $\lambda_1\in W\bullet A_0$ so that $\lambda=\lambda_1\star (-\mu)$. If $\lambda_1=x\bullet \lambda_0$ for $x\in W$ and $\lambda_0\in A_0$, the above
 is equivalent to 
\[\lambda=\lambda_1\star (-\mu)=(x\bullet\lambda_0)\star (-\mu)=x\bullet((-\mu)\bullet \lambda_0)=-px(\mu)+x\bullet \lambda_0.\]
Without loss of generality, we assume $-\mu$ increases $\lambda_1$. That is,  $-x\mu\in \Lambda^+$. 
We have \begin{align*}
\fstab^k_{-\frac{\lambda+\rho}{p}}(y)=&\fstab^k_{x(\mu)-\frac{\lambda_1+\rho}{p}}(y)\\
\overset{\eqref{eqn:Z4}}=&e^{-yx(\mu)}\calL_{x(\mu)}\otimes\fstab^k_{-\frac{\lambda_1+\rho}{p}}(y)\\
\overset{Thm \ref{thm:BR}}=&e^{-yx(\mu)}J_{x(\mu)}^{\oR}e^{-\rho}\gamma^{\lambda_1}Z^\fb(y\bullet\lambda_1 +2\rho)\\
\overset{\eqref{eqn:L_Theta}}=&e^{-yx(\mu)}e^{-\rho}J_{x(\mu)}^{\oR}\gamma^{\lambda_0}(\Theta_{\wt x}^{\oR})^{-1}Z^\fb(y\bullet\lambda_1+2\rho)\\
\overset{Thm \ref{thm:loc_wall}(ii)}=&e^{-yx(\mu)}e^{-\rho}\gamma^{\lambda_0}\Theta^{\oR}_{x(\mu)}(\Theta_{\wt x}^{\oR})^{-1}Z^\fb(y\bullet\lambda_1+2\rho)\\
=&e^{-yx(\mu)}e^{-\rho}\gamma^{\lambda_0-px(\mu)}(\Theta_{\wt x}^{\oR})^{-1}Z^\fb(y\bullet\lambda_1+2\rho)\\
=&e^{-yx(\mu)}e^{-\rho}\gamma^{\lambda}Z^\fb(y\bullet\lambda_1+2\rho).
\end{align*}
Here  the third equality follows from Theorem \ref{thm:BR} and the second step. The second to last equality follows from the fact that $-x\mu$ increases $\lambda_0$, since $x\mu\in \Lambda^-$ and that the inverse of $x\mu$ is $-x\mu$; The last equality follows from the fact that $\lambda\star x^{-1}=\lambda_0-px(\mu)$, and that $x^{-1}$ increases $\lambda$, which in turn is equivalent to $x^{-1}$ increases $\lambda_0\star x$. Lastly, note  that the definition of the Verma modules only depends on an element in $\fh^*_k$, which in particular is invariant under shifting by $p\lambda'$ for any $\lambda'\in \Lambda$. Hence , $Z^\fb(y\bullet\lambda_1+2\rho)$ and $ Z^\fb(y\bullet\lambda+2\rho)$ only differ by a $A_k$-grading, namely, $e^{yx(\mu)}$. Therefore, The proof is finished.
\end{proof}

\subsection{Restriction formula and Lie algebra cohomology}\label{subsec:LieAlgCoh}

In general, for any variety $X$, let $a\in X$ be a closed point with residue field $k_a$ and  embedding $i:\{a\}\inj X$, and $\calF\in D^{\ob}\Coh(X)$. Then $\calF|_a^\vee\cong \Hom_{k_a}(i^*\calF,k_a)\cong \Hom_X(\calF,i_*k_a)=\Hom_X(\calF,\calO_a)$.

We have $k_{x\fb}=\fL^{\lambda} Z_0^{x\fb}(\lambda+2\rho)$  for any $x\in W$. Recall that here $\fb$ is labeled by $id\in W$.
Putting these together with Theorems \ref{thm:IntStable} and \ref{Thm:cat}, we get \[e^{\rho-x\rho}\stab^{+,T\calB,-\frac{\lambda+\rho}{p}}_y|_{x\fb}^\vee=\Ext^*(Z^{\fb}(y\bullet \lambda+2\rho),  Z_0^{x\fb}(\lambda+2\rho)).\]
Here the $\Ext$ is taken in the Grothendieck group of the category $\Mod^{\gr}(U(\fg_k)_0^{\lambda_0},A_k)$. 

\begin{rem}
	This is a reinterpretation of the restriction formula of stable bases in terms of Lie algebra cohomologies. Therefore, the formula \cite[Theorem~7.5]{SZZ} gives the a formula for 
	$\Ext^*(Z^{\fb}(y\bullet\lambda+2\rho),  Z_0^{x\fb}(\lambda+2\rho))^\vee$. In particular, this yields an explicit expression of the Koszul gradings on the Verma modules in terms of those on the baby Verma modules. 
\end{rem}

\begin{rem}\label{rmk:other_localizations}
Let $\omega\in\Omega$, and $\dot{\mu}:=\mu\star \omega$ then
\[\Theta_{\omega}^{-1}(Z^{\fb}_0(\mu+2\rho))\cong Z^{\fb}_0(\dot{\mu}+2\rho),\]
\[\Theta_{\omega}^{-1}(Z^{\fb}(\mu+2\rho))\cong Z^{\fb}(\dot{\mu}+2\rho).\]
The proof is similar to that of Lemma~\ref{lem:babyV&T}, using \cite[\S~2.3.1(1)]{BMR2} in place of \cite[Theorem~2.2.8]{BMR2}, again  taking into account the correction in \cite{BR} in Remark~2.5.2 and the paragraph afterwards.
Together with Lemma~\ref{lem:babyV&T}, one can see that these localization functors have the properties that $\gamma^\lambda Z_0^\fb(\lambda+2\rho)\cong k_{id}$, the latter being the skyscraper sheaf, and $\gamma^\lambda Z^\fb(\lambda+2\rho)\cong 
\calO_{T^*_{id}\calB^{(1)}}$. 
\end{rem}
The right hand side of Theorem~\ref{Thm:cat} is the standard object in the $t$-structure associated to the alcove containing $\lambda$ \cite[\S~1.8.2]{BM}.

\subsection{Other chambers}
For an arbitrary chamber $\fC$, let $B'_{\bbZ}$ be the Borel subgroup whose roots are in $\fC$. Then, identifying the maximal torus $A_\bbZ$ with the abstract Cartan group using $B'_\bbZ$, and  defining the Verma modules of $U(\fg_k)$ using $B'_k$, these Verma modules under localization functor give the categorification of stable basis associated to the chamber $\fC$.

The change of chambers are controlled by the Weyl group as usual. More precisely,
for any Weyl group element $w\in W$, taking a representative of it in $N_{G_{\bbZ}}(A_\bbZ)$, we get an automorphism $w:T_{\bbZ}\to T_{\bbZ}$ of groups together with an automorphism $w:T^*\calB_{\bbZ}\to T^*\calB_{\bbZ}$ of varieties. These two automorphisms intertwine the two actions of $T_{\bbZ}$ on $T^*\calB_{\bbZ}$, hence induces an auto-equivalence of categories $w:D^{\ob}_{T_\bbZ}\Coh(T^*\calB_\bbZ)\to D^{\ob}_{T_\bbZ}\Coh(T^*\calB_\bbZ)$. 
Applying this functor to the objects $\{\fstab_\lambda^\bbZ(y)\mid y\in W\}$ for $\lambda\in\Lambda$ regular, according to \S~\ref{subsec:WeylAction}, we obtain integral forms of the categorification of the stable basis associated to the chamber $w\fC_+$ and the polarization $w(T\calB)=T\calB\in K_{T_{\bbZ}}(T^*\calB_{\bbZ})$.


\section{Monodromy of the quantum cohomology}

The action of the affine braid group discussed in the previous section is one of the many incarnations of the same representation. In this section we briefly review some of the other incarnations, with emphasis on its appearance as the monodromy of the $G\times \bbC^*$-equivariant quantum connection of $T^*\calB$.

\subsection{Other incarnations of the braid group action}

\begin{enumerate}
	\item As has been reviewed in detail, in  \cite{BMR2} an action of $B_{\aff}'$ on $D^{\ob}(\Mod_{A\times\bbC^*}\fg_{\overline \bbF_p})$ has been constructed. 
	\item Via the localization of \cite{BM,BMR1},  $D^{\ob}(\Mod_{A\times\bbC^*}\fg_{\overline \bbF_p})\cong D^{\ob}_{A\times\bbC^*}(T^*\calB)$, the above affine braid group action lifts to an action on $D^{\ob}_{A\times\bbC^*}(T^*\calB)$. 
	\item In \cite{BR} the Fourier--Mukai kernels on $D^{\ob}_{A\times\bbC^*}(T^*\calB)$ on which the generators of $B_{\aff}'$ acts has been determined. Hence, taking Grothendieck group, this action agrees with the affine Hecke algebra $\bbH$ action on $K_{A\times\bbC^*}(T^*\calB)$ of Kazhdan--Lusztig and Ginzburg.
	\item In \cite{BMO}, building up on earlier work of Cherednik \cite{C}, the monodromy of the quantum connection of $T^*\calB$, which a priori is an action of the affine braid group on $H^*_{G\times\bbC^*}(T^*\calB)$, also agrees with the action of $\bbH$ on $K_{A\times\bbC^*}(T^*\calB)$ via proper identifications. More details will be recalled below.
	\item Were the category $D^{\ob}_{A\times\bbC^*}(T^*\calB)$ replaced by $D^{\ob}_{G\times\bbC^*}(T^*\calB)$, then the equivalences of \cite{B16, ArkhBez} give $D^{\ob}_{G\times\bbC^*}(T^*\calB)\cong D^I_{const}(G^L_K/G^L_\calO)$, where $K=\bbF_p(\!(t)\!)$, $\calO=\bbF_p[[t]]$ is the ring of integers, $G^L$ is the dual group, and $I$ is the Iwahori subgroup. The category $D^I_{const}(G^L_K/G^L_\calO)$ categorifies the action of the Iwahori-Matsumoto Hecke algebra on its (anti)spherical module. For $D^{\ob}_{A\times\bbC^*}(T^*\calB)$, the corresponding Langlands dual should be $D^I_{const}(G^L_K/A^L_\calO N^L_K)$, which is beyond our present method. Nevertheless, in \cite{SZZ} we considered its decategorification, which is  the Iwahori-invariants $\Pi^I$ in the principal series representation of $G^L_K$. This is the periodic module of the Iwahori-Matsumoto Hecke algebra. Under the isomorphism $K_{A\times\bbC^*}(T^*\calB)\cong \Pi^I$ of Hecke-modules, we proved that the $K$-theory stable basis goes to the normalized characteristic functions of the orbits, and the basis of $A$-fixed points goes to the Casselman basis \cite{C80}. 
\end{enumerate}

The result in the present paper can be seen as a contribution to yet another realization of the same affine braid group action, i.e., it comes from wall-crossings of the $K$-theory stable bases of $T^*\calB$. Consequently, wall-crossings of the $K$-theory stable bases of $T^*\calB$ are equal to the monodromy of the quantum connections. This is in line with the results of \cite{AO} that the monodromy of the quantum difference equations of Nakajima quiver varieties is equal to the wall-crossing of the elliptic stable envelope.

Now we recall the details of (4) from the above list, following \cite{BMO}.

The quantum connection of $T^*\calB$ is a flat connection on the trivial bundle $H_{G\times\bbC^*}^*(T^*\calB)\times A^\vee_{reg}$ on the dual torus $A^\vee$ with the root hyperplanes removed. This connection is equivariant with respect to the Weyl group  action, hence descents to a flat connection on $A^\vee_{reg}/W$.

On the other hand, by the work of Lusztig \cite{L}, $H^*_{G\times\bbC^*}(T^*\calB)$ can be realized as the polynomial representation $\calM_{\xi,t}$ of the graded affine Hecke algebra, where $\xi$ is the $\bbC^*$-equivariant parameter and $t\in \fh^*$ is the $A$-equivariant parameter. For any module of the graded affine Hecke algebra, there is an associated affine Knizhnik--Zamolodchikov connection on the trivial bundle on $A^\vee_{reg}$, so that derivatives with respect to $\Lie(A^\vee)=\fh^\vee$ are expressed in terms of operators in the graded affine Hecke algebra. Cherednik has studied the monodromy of such a flat connection \cite{C}. For any module $M$ of the graded affine Hecke algebra, the monodromy $\calI M$, which a priori is a representation of $\pi_1(A^\vee_{reg}/W)\cong B_{\aff}'$  on the fiber of $M$, factors through the affine Hecke algebra $\bbH$. 

In \cite{BMO}, these two flat connections have been identified. In other words, quantum multiplication by a divisor class in $H^2(T^*\calB,\bbC)\cong \fh^\vee$ is expressed as taking derivatives with respect to  the affine Knizhnik--Zamolodchikov connection. Consequently, it follows from the work of Cherednik that the monodromy $\calI (\calM_{\xi,t})$ is the polynomial representation $M_{q,z}$ of the affine Hecke algebra $\bbH$ with $q=\exp(\xi)$ and $z=\exp(t)$. Moreover, under the natural isomorphism $M_{q,z}\cong K_{G\times\bbC^*}(T^*\calB)$, the action of $\bbH$ on $M_{q,z}$ coming from the monodromy is further identified with the action on $K_{G\times\bbC^*}(T^*\calB)$ of Kazhdan--Lusztig \cite{KL} and Ginzburg \cite{CG}. 
Extending scalars of the module $M_{q,z}$ via the map $K_{G\times\bbC^*}(\pt)\to K_{A\times\bbC^*}(\pt)$, we get an action of $B'_{\aff}$ on $K_{A\times\bbC}(T*\calB)$, which factors through the $\bbH$-action considered in \S~\ref{sec:recoll}.
In other words, for any $b\in \pi_1(A^\vee_{reg}/W)\cong B_{\aff}'$,  the monodromy operator $b:\calI (\calM_{\xi,t})\to \calI (\calM_{\xi,t})$ is given by $\calI (\calM_{\xi,t})\cong K_{G\times\bbC^*}(T^*\calB)\to K_{G\times\bbC^*}(T^*\calB)\cong \calI (\calM_{\xi,t})$,
the image of $b$ under the projection $\bbC[q^\pm][B'_{\aff}]\to\bbH$.

\subsection{K-theory wall-crossing = cohomological monodromy}\label{subsec:wallC-rossingMonodromy}
Note that $\fh^\vee_{\bbC,reg}/(W_{\aff},\bullet)$ is naturally a covering space of $A^\vee_{reg}/W$. Let $\na_1$ and $\na_2$ be two alcoves in $\fh^\vee_{\bbC,reg}$, which are separated by a wall, and $\na_1$ is below $\na_2$ in the natural order determined by that on $\Lambda$. Note that this order has been used in Remark~\ref{rem:Left_right_J}(a), and is an extension of the order used in Theorem \ref{thm:nonsimplewall} for $\fC_+$. Then, the  homotopy class of a positive  path in $\fh^\vee_{\bbC,reg}$, going from a point in $\na_1$ to a point in $\na_2$ in the same $W_{\aff}$ orbit, determines a unique element \cite[Lemma~1.8.1]{BM}
\[b_{\na_2,\na_1}\in \pi_1(\fh^\vee_{\bbC,reg}/(W_{\aff},\bullet))\subseteq \pi_1(A^\vee_{reg}/W)\cong B_{\aff}'.\]
Note that we have $b_{\nabla_3,\nabla_1}=b_{\nabla_3,\nabla_2}b_{\nabla_2,\nabla_1}$, $b_{x\na_-, \na_-}=\tilde{x}$ for any $x\in W$, 
and $b_{x\nabla_-+\lambda,\nabla_-}=\lambda\wt{x}$ for any $\lambda\in\Lambda$ and $x\in W$ \cite[Lemma~1.8.1]{BM}. The subscripts of $b$ here we use are again  different from that of \cite{BM} for the same reason as explained Remark~\ref{rem:Left_right_J}(a).

To relate the affine braid group action of Bezrukavnikov and Riche with the monodromy braid group action, we first define the automorphism $\tau:B'_{\aff}\to B'_{\aff}$, $\wt{s_\alpha}\mapsto \wt{s_\alpha}$ for $\alpha\in \Sigma$, and $\lambda\mapsto-\lambda$ for $\lambda\in\Lambda$. 
Recall that the action of $B_{\aff}$, denoted by $J^{\oR}$, is such that $J^{\oR}_{ab}=J^{\oR}_b\circ J^{\oR}_a$. Hence, in case $b_{x\na_-+\la, \na_-}=\la \wt{x}$, then 
\begin{align*}
(J^{\oR}_{\tau(b_{x\nabla_-+\lambda,\nabla_-})})^{-1}&=(J^{\oR}_{\tau(\lambda\wt{x})})^{-1}\\
&=(J^{\oR}_{\wt{x}}\circ J^{\oR}_{-\lambda})^{-1}\\
&=J^{\oR}_{\lambda}\circ(J^{\oR}_{\wt{x}})^{-1}
\end{align*}

For simplicity, the  class in  the Grothendieck group of the auto-equivalence $J^{\oR}_{\tau(b)}:D^{\ob}_{T_\bbZ}(T^*\calB_\bbZ)\to D^{\ob}_{T_\bbZ}(T^*\calB_\bbZ)$ for $b\in\ B_{\aff}$ is denoted by $[b]$.
\begin{thm}\label{thm:monodromystable} Notations as above.
For each pair of alcoves $\na_1$ and $\na_2$, the monodromy action of $b_{\na_2,\na_1}\in \pi_1(A^\vee_{reg}/W)\cong B_{\aff}'$:
\[[b_{\na_2,\na_1}^{-1}]:\calI(\calM_{\xi,t})\otimes_{K_{G\times\bbC^*}(pt)}K_{A\times\bbC^*}(pt)\cong K_{A\times\bbC^*}(T^*\calB)\to K_{A\times\bbC^*}(T^*\calB)\cong\calI(\calM_{\xi,t})\otimes_{K_{G\times\bbC^*}(pt)}K_{A\times\bbC^*}(pt)\] 
sends the set $\{\calL_{-\rho}\otimes\stab^{+,T\calB,\na_1}_w\mid w\in W\}$ to $\{\calL_{-\rho}\otimes\stab^{+,T\calB,\na_2}_w\mid w\in W\}$, up to (explicitly determined) scalars. 

More precisely, assume $\nabla_1=x\nabla_-+\lambda$ and $\na_2=y\na_-+\mu$, with $x,y\in W$ and $\lambda,\mu\in Q$. Then, $[b_{\na_2,\na_1}^{-1}]$ sends $\calL_{-\rho}\otimes\stab^{+,T\calB,\na_1}_w$ to $e^{w\lambda-wxy^{-1}\mu}\calL_{-\rho}\otimes\stab^{+,T\calB,\na_1}_{wxy^{-1}}$.
\end{thm}

For example, in the case when  alcoves $\na_1$ and $\na_2$ are  adjacent and separated by a wall $H_{x\dot \alpha^\vee,0}$ for some $x\in W$ and $\alpha\in \Sigma$, i.e., there exists $\lambda\in \Lambda$ and $x\in W$, such that $\na_1=x\na_-$ and $\na_2=xs_\alpha\na_-$. 
Then $b_{\na_2, \na_1}=\wt{xs_\alpha}\wt{x}^{-1}$, and 
the corresponding operator  is
\[J^{\oR}_{\tau(b_{\na_2, \na_1})}=J^{\oR}_{\wt{x}^{-1}}J^{\oR}_{\wt{xs_\alpha}}= \left(J^{\oR}_{\wt{x}}\right)^{-1}J^{\oR}_{\wt{xs_\alpha}}.\] 
\begin{proof}
	Let $\nabla_1=x\nabla_-+\lambda$. Then, a direct consequence of the definition in \S \ref{sec:intst} is that $\fstab_{\na_1}^\bbZ(w)=e^{-w\lambda}J^{\oR}_\lambda(J^{\oR}_{\tilde{x}})^{-1}\fstab_{\na_-}^\bbZ(wx)$.
	Hence, induced K-theory map of the functor $(J^{\oR}_{\tau(b_{x\nabla_-+\lambda,\nabla_-})})^{-1}$ sends $\calL_{-\rho}\otimes\stab^{+,T\calB,\na_-}_{wx}$ to $e^{-w\lambda}\calL_{-\rho}\otimes\stab^{+,T\calB,\na_1}_w$ by  Theorem~\ref{thm:IntStable}.

The analysis above implies that $[b_{\nabla_1,\nabla_-}^{-1}]$, which is the induced K-theory map  of the functor $(J^{\oR}_{\tau(b_{x\nabla_-+\lambda,\nabla_-})})^{-1}$, sends
$\calL_{-\rho}\otimes\stab^{+,T\calB,\na_-}_{wx}$ to $e^{-w\lambda}\calL_{-\rho}\otimes\stab^{+,T\calB,\na_1}_w$.
 Similarly, $[b_{\nabla_2,\nabla_-}^{-1}]$ which is the induced K-theory map  of the functor  $(J^{\oR}_{\tau(b_{y\nabla_-+\mu,\nabla_-})})^{-1}$, sends $\calL_{-\rho}\otimes\stab^{+,T\calB,\na_-}_{vy}$ to $e^{-v\mu}\calL_{-\rho}\otimes\stab^{+,T\calB,\na_2}_v$. 
Hence, $[b_{\na_2,\na_1}^{-1}]=[b_{\nabla_2,\nabla_-}^{-1}][b_{\nabla_1,\nabla_-}]$ sends $\calL_{-\rho}\otimes\stab^{+,T\calB,\na_1}_w$ to $e^{w\lambda-wxy^{-1}\mu}\calL_{-\rho}\otimes\stab^{+,T\calB,\na_1}_{wxy^{-1}}$.
\end{proof}

Much more general phenomenon has been expected for symplectic resolutions. Namely, for a symplectic resolution $X$, there are derived auto-equivalences of $D^{\ob}(X)$ that come from quantizations of $X$ \cite{Kaledin}. These derived auto-equivalences are expected to be related to the monodromy of the quantum connection of $X$ \cite{O15, BO}. For quiver varieties, such results are proven by Bezrukavnikov and Okounkov \cite{BO}. In future works we will explore in more examples of symplectic resolutions  the categorification of the $K$-theoretic stable bases in terms of standard objects of the quantizations.

\section{Appendix: Wall-crossings in the example of $\SL(3,\bbC)$}\label{appendix}
In this appendix, we compute some wall-crossing matrix coefficients when $G=\SL(3,\bbC)$.

\subsection{Crossing a simple wall}\label{sec:examsl3} There are two simple roots $\al_1, \al_2$ and another positive root $\al_3=\al_1+\al_2$. Denote $s_i=s_{\al_i}$. The identity element of $W$ is denoted by $e$. The fundamental weights $\omega_i$ are noted by the  arrows. 
Let $\na_1=s_{1}\na_+, \na_2=\na_+-\varpi_1$, and $\na_3=\na_2+\varpi_2$. See  Figure 1 below.

We consider the dominant chamber $\fC_+$ and polarization $T\calB$, and  compute $f_y^{\na_1\lefta\na_+}$. 

\begin{figure}\label{fig}
	\begin{tikzpicture}[scale=1.2]
	\tikzstyle{every node}=[font=\small]
	
	
	\foreach \x in {-1}{
		\draw (\x+2,-1.7321*2+1.7321) to (3.5, -1.7321*\x+0.5*1.7321);
	}
	\draw (-2+2,-1.7321*2+1.7321) to (2.5, 1.5*1.7321);
	\draw (-3+2,-1.7321*2+1.7321) to (1.5, 1.5*1.7321);
	\draw (-2,-1.7321*1) to (0.5,1.7321*1.5);
	
	\draw (-0.5,-1.7321*(3-2.5) to (-1,0);
	\draw (0.5,-1.7321*(3-2.5) to (-0.5,0.5*1.7321);
	\draw (2.5,-1.7321+1.5*1.7321) to (-1+2,1.7321*2.0);
	\draw (2,-1.7321*2+2*1.7321) to (-2+2,1.7321*2.0);
	\draw (1.5,-1.7321*1.5+1.7321) to (-1,1.7321*2);
	\draw (3,1.7321) to (2.5,1.7321*1.5); 
	
	\draw (-1,1.7321*0.5) to (2.5,1.7321*0.5);
	\draw (0.5,1.7321+1.7321*0.5) to (3.5,1.7321+1.7321*0.5);
	\draw (-1.5,-1.7321+1.7321*0.5) to (1.5,-1.7321+1.7321*0.5);
	\draw (-1,0) to (2,0);
	\draw (-1,1.7321) to (3,1.7321);
	
	
	
	\filldraw (-1,1.7321) 
	node[anchor=south, yshift=-0.1cm,xshift=-0.1cm] {$H_{\alpha_3^\vee,2}$}
	node[anchor=east, yshift=0.2cm, xshift=1.05cm] {$+$}
	node[anchor=east, yshift=-0.2cm, xshift=0.9cm] {$-$};
	\filldraw (-1,1.7321*0.5) 
	node[anchor=south, yshift=-0.1cm, xshift=-0.4cm] {$H_{\alpha_3^\vee,1}$}
	node[anchor=east, yshift=0.2cm, xshift=0.65cm] {$+$}
	node[anchor=east, yshift=-0.2cm, xshift=0.5cm] {$-$};
	\filldraw (-1,0) 
	node[anchor=south, yshift=-0.1cm, xshift=-0.4cm] {$H_{\alpha_3^\vee,0}$};
	
	\filldraw (-1+2,1.7321*2.0) 
	node[anchor=south,xshift=0.5cm] {$H_{\alpha_1^\vee,2}$}
	node[anchor=east,xshift=0.5cm] {$+$}
	node[anchor=east] {$-$};
	\filldraw (-1+1,1.7321*2.0)
	node[anchor=south,xshift=0.3cm] {$H_{\alpha_1^\vee,1}$}
	node[anchor=east,xshift=0.5cm] {$+$}
	node[anchor=east] {$-$};
	\filldraw (-1,1.7321*2.0)
	node[anchor=south,xshift=0.2cm] {$H_{\alpha_1^\vee,0}$}
	node[anchor=east,xshift=0.5cm] {$+$}
	node[anchor=east] {$-$};
	
	\filldraw (0.5,-1.7321*2+1.7321) 
	node[anchor=south, yshift=-0.6cm, xshift=-0.2cm] {$H_{\alpha_2^\vee,0}$};
	\filldraw (-0.5,-1.7321*2+1.7321) 
	node[anchor=south, yshift=-0.6cm, xshift=-0.4cm] {$H_{\alpha_2^\vee,1}$}
	node[anchor=east, yshift=0.2cm, xshift=-0.5cm] {$+$}
	node[anchor=east, yshift=0.1cm, xshift=0.1cm] {$-$}; 
	\filldraw (-1,-1.7321*2+1.7321) 
	node[anchor=south, yshift=-0.6cm, xshift=-1.2cm] {$H_{\alpha_2^\vee,2}$}
	node[anchor=east, yshift=0.2cm, xshift=-1.2cm] {$+$}
	node[anchor=east, yshift=0.1cm, xshift=-0.6cm] {$-$}
	;
	\node[anchor=south,yshift=-0.4cm,xshift=0.6cm] at (0, 0.25*1.7321){\tiny{$\na_1$}}; 
	\node[anchor=south,yshift=0.3cm,xshift=0.6cm] at (0.5, 0.05*1.7321){\tiny{$\na_+$}}; 
	\node[anchor=south,yshift=-0.6cm,xshift=0.6cm] at (-0.5, 0.5*1.7321){\tiny{$\na_3$}}; 
	\node[anchor=south,yshift=-0.4cm,xshift=0.6cm] at (0, -0.15*1.7321){\tiny{$\na_2$}}; 
	\node[anchor=south,yshift=-0.4cm,xshift=0.6cm] at (0.5, 0){\tiny{0}}; 
	\node[anchor=south,yshift=-1cm,xshift=0.6cm] at (0.5, 0){\tiny{$\na_-$}}; 
	\node[anchor=north,yshift=0.5cm,xshift=0.6cm] at (0, 0.5*1.7321){\tiny{$\omega_2$}}; 
	\node[anchor=north,yshift=0.5cm,xshift=0.6cm] at (1, 0.5*1.7321){\tiny{$\omega_1$}}; 
	\node[xshift=0.6cm] at (0.5, 0){$\bullet$}; 
	\draw[->, thick](1,0) to (0.5, 0.5*1.7321);
	\draw[->, thick](1,0) to (1.5, 0.5*1.7321);
	\end{tikzpicture}
	\caption{Alcoves of $SL(3,\bbC)$. }
\end{figure}

We first cross $H_{\alpha_3^\vee, 0}$, i.e., $\na_1\rightarrow \na_2$. There are two ways to do this: $\na_1\rightarrow \na_+\rightarrow \na_2$ and $\na_1\rightarrow \na_3\rightarrow \na_2$. The second steps of both ways are done by translations. Then by Lemma \ref{lem:coefficients} and  \eqref{eq:shift}, we have the following formulas corresponding to the two ways.
\begin{align}\label{eq:example1}
\stp{\na_1}{y}&=\left\{\begin{array}{ll}
\stp{\na_+}{y}+f_y^{\na_1\leftarrow \na_+}\stp{\na_+}{ys_{1}}, & \textit{ if\quad} ys_{1}<y,\\
\stp{\na_+}{y}, &  \textit{ if\quad} ys_{1}>y. 
\end{array}\right.\nonumber\\
&=\left\{\begin{array}{ll}
e^{-y\varpi_1}\calL_{\varpi_1}\otimes\stp{\na_2}{y}+e^{-ys_{1}\varpi_1}f^{\na_1\leftarrow \na_+}_y\calL_{\varpi_1}\otimes\stp{\na_2}{ys_{1}}, & \textit{ if\quad} ys_{1}<y,\\
e^{-y\varpi_1}\calL_{\varpi_1}\otimes\stp{\na_2}{y}, &  \textit{ if\quad} ys_{1}>y.
\end{array}\right.
\end{align}
\begin{align}\label{eq:example2}
\stp{\na_1}{y}&=\left\{\begin{array}{ll}
\stp{\na_3}{y}+f_y^{\na_1\leftarrow \na_3}\stp{\na_3}{ys_{2}}, & \textit{ if\quad} ys_{2}<y,\\
\stp{\na_3}{y}, &  \textit{ if\quad} ys_{2}>y.
\end{array}\right.\nonumber\\
&=\left\{\begin{array}{ll}
e^{-y\varpi_2}\calL_{\varpi_2}\otimes\stp{\na_2}{y}+e^{-ys_{2}\varpi_2}f_y^{\na_1\leftarrow \na_3}\calL_{\varpi_2}\otimes\stp{\na_2}{ys_{2}}, & \textit{ if\quad} ys_{2}<y,\\
e^{-y\varpi_2}\calL_{\varpi_2}\otimes\stp{\na_2}{y}, &  \textit{ if\quad} ys_{2}>y.
\end{array}\right.
\end{align}
Comparing \eqref{eq:example1} and \eqref{eq:example2} with $y=s_{1}$,  and restricting to the fixed point $e$,  we get
\begin{align}
\label{eq:SL3}(1-e^{\alpha_1})\stp{\na_2}{s_{1}}|_e=f_{s_{1}}^{\na_1\leftarrow \na_+}\stp{\na_2}{e}|_e.
\end{align}
By the normalization formula \eqref{eq:normp},  \[\stp{\na_2}{e}|_e=(1-e^{\alpha_1})(1-e^{\alpha_2})(1-e^{\alpha_3}).\]
By \eqref{eq:shift} and Lemma \ref{lem:sheafdual}, \[\stp{\na_2}{s_{1}}|_e=e^{-\alpha_1}\stp{\na_+}{s_{1}}|_e=-e^{-\alpha_1+2\rho}(\stp{\na_-}{s_{1}})^\vee|_e=-e^{\alpha_2+\alpha_3}(\stp{\na_-}{s_{1}})^\vee|_e.\]
Due to Lemma \ref{lem:sheafdual} and \cite[Proposition 3.6(3)]{SZZ}, if $ws_{\alpha_i}>w$, we have the recursion for any $u\in W$:
\begin{equation}\label{eq:recursion}
q^{{1/2}}\stp{\na_-}{ws_{\alpha_i}}|_u=\frac{q-1}{1-e^{u\alpha_i}}\stp{\na_-}{w}|_u-\frac{e^{u\alpha_i}-q}{1-e^{-u\alpha_i}}\stp{\na_-}{w}|_{us_{\alpha_i}}. 
\end{equation}
Let $w=u=e$ and $\alpha_i=\alpha_1$, we get
\[\stp{\na_-}{s_{1}}|_e=(q^{{1/2}}-q^{-{1/2}})(1-e^{\alpha_2})(1-e^{\alpha_3}).\]
Therefore, 
\[\stp{\na_2}{s_1}|_e=-e^{\alpha_2+\alpha_3}(q^{-{1/2}}-q^{{1/2}})(1-e^{-\alpha_2})(1-e^{-\alpha_3})=(q^{{1/2}}-q^{-{1/2}})(1-e^{\alpha_2})(1-e^{\alpha_3}).\]
Hence, sovling from \eqref{eq:SL3}, we get
\[f_{s_1}^{\na_1\leftarrow \na_+}=q^{{1/2}}-q^{-{1/2}}.\]

By  Lemma \ref{lem:shift}, we also have $f_{y}^{\na_-\lefta s_1\na_-}=f_y^{\na_1\lefta\na_+}$.

\subsection{ Crossing a non-simple wall}\label{sec:continuedexe}
We calculate the constants $f_{s_1s_2}^{\na_2\leftarrow \na_1}$, $f_{s_2s_1}^{\na_2\leftarrow \na_1}$ and $f_{s_3}^{\na_2\leftarrow \na_1}$.
Let $y=s_1s_2$ in \eqref{eq:example1}, we get
\[\stp{\na_1}{s_1s_2}=e^{-s_1s_2\varpi_1}\calL_{\varpi_1}\otimes\stp{\na_2}{s_1s_2}.\]
By Lemma \ref{lem:coefficients} and the identity $f_{s_1s_2}^{\na_2\lefta\na_1}=-f_{s_1s_2}^{\na_1\lefta\na_2}$, we have
\[\stp{\na_1}{s_1s_2}=\stp{\na_2}{s_1s_2}-f_{s_1s_2}^{\na_2\leftarrow \na_1}\stp{\na_2}{s_2}.\]
Comparing the two identities and restricting to the fixed point $s_2$, we get
\begin{align}
\label{eq:SL3nonsimple}e^{-s_1s_2\varpi_1+s_2\varpi_1}\stp{\na_2}{s_1s_2}|_{s_2}=e^{\alpha_1}\stp{\na_2}{s_1s_2}|_{s_2}=\stp{\na_2}{s_1s_2}|_{s_2}-f_{s_1s_2}^{\na_2\leftarrow \na_1}\stp{\na_2}{s_2}|_{s_2}.
\end{align}
So we just need to compute $\stp{\na_2}{s_2}|_{s_2}$ and $\stp{\na_2}{s_1s_2}|_{s_2}$. 
By  (\ref{eq:recursion}) with $w=s_1, i=2, u=s_2$, we get
\[\stp{\na_-}{s_1s_2}|_{s_2}=q^{-{1/2}}\frac{q-e^{-\alpha_2}}{1-e^{\alpha_2}}\stp{\na_-}{s_1}|_e=(1-q^{-1})(q-e^{-\alpha_2})(1-e^{\alpha_3}).\]
By Lemma \ref{lem:sheafdual},
\[\stp{\na_+}{s_1s_2}|_{s_2}=e^{\alpha_1}(1-q)(1-q^{-1}e^{-\alpha_2})(e^{\alpha_3}-1).\]
By the translation formula \eqref{eq:shift},
\[\stp{\na_2}{s_1s_2}|_{s_2}=(1-q)(1-q^{-1}e^{-\alpha_2})(e^{\alpha_3}-1).\]
By the normalization formula \eqref{eq:normp}, 
\[\stp{\na_2}{s_2}|_{s_2}=q^{-{1/2}}(1-e^{\alpha_1})(1-e^{\alpha_3})(q-e^{-\alpha_2}).\]
Therefore, by using \eqref{eq:SL3nonsimple}
\[f_{s_1s_2}^{\na_2\leftarrow \na_1}=q^{{1/2}}-q^{-{1/2}}.\]
By symmetry of $\al_1$ and $\al_2$, we have
\[f_{s_2s_1}^{\na_2\leftarrow \na_1}=q^{{1/2}}-q^{-{1/2}}.\]

Let us compute $f_{s_1s_2s_1}^{\na_2\leftarrow \na_1}$. By Lemma \ref{lem:coefficients}, we have
\[\stp{\na_1}{s_1s_2s_1}=\stp{\na_2}{s_1s_2s_1}-f_{s_1s_2s_1}^{\na_2\leftarrow \na_1}\stp{\na_2}{s_1}.\]
Hence, we only need to know $\stp{\na_1}{s_1s_2s_1}|_{e}$ and $\stp{\na_2}{s_1s_2s_1}|_{e}$. We  compute $\stp{\na_-}{s_1s_2s_1}|_{e}$ first. By \eqref{eq:recursion} with $w=s_{1}s_2, \al_i=\al_2,u=e$, we get
\begin{align*}
q^{{1/2}}\stp{\na_-}{s_1s_2s_1}|_e&=\frac{q-1}{1-e^{\alpha_1}}\stp{\na_-}{s_1s_2}|_e-\frac{e^{\alpha_1}-q}{1-e^{-\alpha_1}}\stp{\na_-}{s_1s_2s_1}|_{s_1}\\
&=\frac{q-1}{1-e^{\alpha_1}}q^{-1}(q-1)^2(1-e^{\alpha_3})-\frac{e^{\alpha_1}-q}{1-e^{-\alpha_1}}(1-q^{-1})(1-e^{\alpha_2})(q-e^{-\alpha_1})\\
&=(1-q^{-1})(q^2+1+qe^{\alpha_3}-qe^{\alpha_2}-q-1e^{\alpha_1}).
\end{align*}
By Lemma \ref{lem:sheafdual}, we can compute $\stp{\na_+}{s_1s_2s_1}|_e$, which, together with the formula $\na_2=\na_+-\omega_1$, can be used to get  
\[\stp{\na_2}{s_1s_2s_1}|_e=e^{\alpha_3}(q^{{1/2}}-q^{-{1/2}})(q^{-1}+q+e^{-\alpha_3}-e^{-\alpha_2}-1-e^{-\alpha_1}).\]
Finally, from $\na_1=\na_-+\omega_2$, we can get \[\stp{\na_1}{s_1s_2s_1}|_e=e^{\alpha_3}(q^{{1/2}}-q^{-{1/2}})(q+q^{-1}+e^{\alpha_3}-e^{\alpha_2}-1-e^{\alpha_1}).\]
Therefore,
\begin{align*}
f_{s_2s_1s_2}^{\na_2\leftarrow \na_1}&=\frac{\stp{\na_2}{s_1s_2s_1}|_e-\stp{\na_1}{s_1s_2s_1}|_e}{\stp{\na_1}{e}|_e}\\
&=(q^{{1/2}}-q^{-{1/2}})\frac{e^{\alpha_3}(e^{-\alpha_3}-e^{-\alpha_2}-e^{-\alpha_1}-e^{\alpha_3}+e^{\alpha_2}+e^{\alpha_1})}{(1-e^{\alpha_1})(1-e^{\alpha_2})(1-e^{\alpha_3})}\\
&=(q^{{1/2}}-q^{-{1/2}}).
\end{align*}

\subsubsection*{Acknowledgments}
C.S. benefited a lot from various discussions with R. Bezrukavnikov, A. Okounkov, S. Riche and E. Vasserot. He would also like to thank the I.H.E.S. for providing excellent research environment. C.Z. would like to thank C. Lenart for helpful conversations. A part of the paper was prepared when  G.Z. was affiliated to the Institute of Science and Technology Austria, Hausel Group, supported by
the Advanced Grant Arithmetic and Physics of Higgs moduli spaces No. 320593 of the European Research Council. Another part was prepared when G.Z. was under the support of the Australian Research Council via the award DE190101222. We thank the anonymous referee for many useful suggestions.

\bibliographystyle{halpha}

\newcommand{\arxiv}[1]
{\texttt{\href{http://arxiv.org/abs/#1}{arXiv:#1}}}
\newcommand{\doi}[1]
{\texttt{\href{http://dx.doi.org/#1}{doi:#1}}}
\renewcommand{\MR}[1]
{\href{http://www.ams.org/mathscinet-getitem?mr=#1}{MR#1}}

\end{document}